\documentclass{article}
\usepackage{amsmath}
\usepackage{amssymb}
\usepackage{titlesec}
\usepackage{graphicx}
\usepackage{fancyhdr}
\usepackage{latexsym}
\usepackage{mathrsfs}
\usepackage{booktabs}
\usepackage{mathrsfs}
\usepackage{mathtools}
\usepackage{amsthm}
\usepackage{wasysym}
\usepackage{cite}
\usepackage{color}
\usepackage{amsfonts}
\usepackage{amsmath,amssymb}
\usepackage{graphics}
\usepackage{multimedia}
\usepackage{graphics}
\usepackage{cite}
\usepackage{latexsym}
\usepackage{amsmath}
\usepackage{amssymb}
\usepackage[dvips]{epsfig}
\usepackage{amscd}
\usepackage{amsmath, amsthm}
\numberwithin{equation}{section}
\newtheorem{definition}{Definition}[section]
\newtheorem{theorem}{Theorem}[section]
\newtheorem{lemma}{Lemma}[section]
\newtheorem{remark}{Remark}[section]
\newtheorem{proposition}{Proposition}[section] 

\newtheorem{corollary}{Corollary}[section]

\title{Global spherically symmetric classical solutions for arbitrary large initial data of the multi-dimensional non-isentropic compressible Navier-Stokes equations}
\author{Yongteng G{\small U}$^{a}$\thanks{Email addresses: guyongteng24@mails.ucas.ac.cn (Y. T. Gu), xdhuang@amss.ac.cn (X. D. Huang). }, Xiangdi H{\small UANG}$^{a}$  \\ 
	{\normalsize a. State Key Laboratory of Mathematical Sciences, Academy of Mathematics and Systems Sciences,}\\
	{\normalsize Chinese Academy of Sciences, Beijing 100190, China;}\\
}
\date{}

\begin{document}
	
	\maketitle
	
	\begin{abstract}
		In 1871, Saint-Venant introduced the shallow water equations. Since then, the global classical solutions for arbitrary large initial data of the multi-dimensional viscous Saint-Venant system have remained a well-known open problem. It was only recently that [Huang-Meng-Zhang, http:arXiv:2512.15029, 2025], under the assumption of radial symmetry, first proved the existence of global classical solutions for arbitrary large initial data to the initial-boundary value problem of the two-dimensional viscous shallow water equations. At the same time, [Chen-Zhang-Zhu, http:arXiv:2512.18545, 2025] also independently proved the existence of global large solutions to the Cauchy problem of this system. Notably, in the work of Huang-Meng-Zhang, they also established the existence of global  classical solutions for arbitrary large initial data to the isentropic compressible Navier-Stokes equations satisfying the BD entropy equality in both two and three dimensions, and the viscous shallow water equations are precisely a specific class of isentropic compressible fluids subject to the BD entropy equality. In this paper, we prove a new BD entropy inequality for a class of non-isentropic compressible fluids, which can be regarded as a generalization of the shallow water equations with transported entropy. Employing new estimates on the lower bound of density different from that of Huang-Meng-Zhang's work, we show the "viscous shallow water system with transport entropy" will admit global classical solutions for arbitrary large initial data to the spherically symmetric initial-boundary value problem in both two and three dimensions. Our results also relax the restrictions on the dimension and adiabatic index imposed in Huang-Meng-Zhang's work on the shallow water equations, extending the range from $N=2,\ \gamma \ge \frac{3}{2}$ to $N=2,\ \gamma > 1$ and $N=3,\  1<\gamma<3$.
	\end{abstract}
	\section{Introduction}
	In this paper, we consider the non-isentropic compressible Navier-Stokes equations, especially for viscous shallow water model. This model can be regarded as an extension of the viscous shallow water model to the case of non-constant entropy. It describes the motion of a compressible fluid with variable entropy, and can be viewed as the full compressible Navier-Stokes-Fourier neglecting thermal conduction and viscous dissipation. This generalized shallow water model can be written as
	\begin{equation}\label{1.1}
		\begin{cases}
			\rho _t+\nabla \cdot \left( \rho \mathbf{u} \right) =0,\\
			\rho \dot{\mathbf{u}}+\nabla P=2\nabla \cdot \left( \rho \mathcal{D} \mathbf{u} \right) ,\\
			s_t+u\cdot \nabla s=0,\\
			P=\rho ^{\gamma}e^s.\\
		\end{cases}
	\end{equation}
	Here $\rho(x,t)$, $\mathbf{u}(x,t)$, $s(x,t)$ and $P(\rho,s)$($\gamma>1$) are the fluid density, velocity, entropy and pressure, respectively, and $\mathcal{D} \mathbf{u} = \frac{\nabla \mathbf{u} + \nabla \mathbf{u}^T}{2}$.
	\par
	When $s=\text{const}$, this model is a special case of the compressible Navier-Stokes equations with density-dependent viscosity coefficients
	\begin{equation}
		\begin{cases}
			\partial_t \rho + \nabla \cdot (\rho \mathbf{u}) = 0, \\[4pt]
			\partial_t (\rho \mathbf{u}) + \nabla \cdot (\rho \mathbf{u} \otimes \mathbf{u}) + \nabla P(\rho) - \nabla \cdot \big( \mu(\rho) \, \mathcal{D}(\mathbf{u}) \big) - \nabla \big( \lambda(\rho) \, \nabla \cdot \mathbf{u} \big) = 0,
		\end{cases}
	\end{equation}
	with the viscosity coefficients $\mu(\rho)$ and $\lambda(\rho)$ satisfy 
	$$\mu(\rho) \ge 0, \quad \mu(\rho) + N\lambda(\rho) \ge 0.$$
	Among compressible Navier–Stokes equations with variable viscosity coefficients, a model that has attracted considerable attention is the case $$\mu = \text{const} > 0,\quad\lambda(\rho) = \rho^\beta.$$
	An early study of this case was carried out by Kazhikhov-Vaigant\cite{Kazhikhov-Vaigant}, who proved the global well‑posedness of strong solutions to the two‑dimensional periodic problem for arbitrarily large initial data away from vacuum, provided the exponent satisfies $\beta>3$. Later, by employing commutator estimates, Huang-Li\cite{Huang-Li-1} weakened the requirement on the viscosity coefficient to $\beta>\frac{4}{3}$ and allowed the initial data to contain vacuum. For the Cauchy problem, Huang-Li\cite{Huang-Li-2} established the global existence and uniqueness of both strong and classical solutions under the condition $\beta > \frac{4}{3}$, provided the initial density $\rho_0$ decays at a suitable rate (a related result was also obtained by Jiu-Wang-Xin\cite{Jiu-Wang-Xin}). For the initial-boundary value problem with Navier‑slip conditions, Fan-Li-Li\cite{Fan-Li-Li} employed the Riemann mapping theorem and the pull‑back Green’s function method to overcome the boundary‑induced difficulties and established the well‑posedness of the system. For the radially symmetric case, Huang-Su-Yan-Yu\cite{Huang-Su-Yan-Yu} derived an explicit representation for the effective viscous flux on the boundary and proved the global existence of strong solutions under the condition $\beta>1$. 
	\par 
	Another important variable-viscosity model is the B-D entropy model, in which the coefficients satisfy 
	$$\lambda(\rho) = \rho \mu'(\rho) - \mu(\rho).$$
	This mathematical entropy estimate was derived by Bresch-Desjardins-Lin\cite{Bresch-Desjardins-Lin} and Bresch-Desjardins\cite{Bresch-Desjardins}. Subsequently, by establishing a new a priori bound for smooth approximate solutions, Mellet-Vasseur\cite{Mellet-Vasseur} investigated the stability properties. In the context of multi-dimensional flows, a breakthrough was achieved independently by Li-Xin\cite{Li-Xin}  and Vasseur-Yu\cite{Vasseur-Yu}; they both successfully constructed global weak solutions for the specific case $\mu(\rho) = \rho$ and $\lambda(\rho) = 0$, allowing for vacuum and arbitrarily large initial data. More precisely, they proved global weak solutions for two cases
	\begin{itemize}
		\item $N=2$, $\gamma>1.$ 
		\item $N=3$, $1<\gamma<3.$
	\end{itemize}
	Later, Bresch-Vasseur-Yu\cite{Bresch-Vasseur-Yu} extended the result of Li-Xin\cite{Li-Xin} to the three-dimensional case with a physical symmetric viscous stress tensor and more general viscosity coefficients. For spherically symmetric flows, the global existence of non-vacuum strong solutions with large data was first derived by Zhang\cite{Zhang}, who considered the specific BD entropy relation
	$$\mu(\rho) = \rho^\alpha, \quad \lambda(\rho) = (\alpha - 1)\rho^\alpha.$$
	This model was further studied by Guo-Xu-Zhang\cite{Guo-XU-Zhang} , who constructed global strong solutions for the Cauchy problem. A particularly important case satisfying the BD relation is given by $\mu(\rho) = \rho$ and $\lambda(\rho) = 0$, which corresponds to the viscous Saint-Venant model for shallow water motion. This system has yielded significant results regarding strong solutions. For non-vacuum initial data, Haspot\cite{Haspot} and Burtea-Haspot \cite{Burtea-Haspot} established the global well-posedness of strong solutions for the one-dimensional Cauchy problem. Subsequently, Cao-Li-Zhu\cite{Cao-Li-Zhu}  derived the existence of global regular solutions for the one-dimensional case with large data and far-field vacuum. In a later work \cite{Cao-Li-Zhu-2}, they further investigated the initial-boundary value problem in the exterior of a ball in $\mathbb{R}^d$ ($d=2,3$), proving the global existence of unique spherically symmetric classical solutions for large initial data with far-field vacuum.
	However, the existence of global classical solutions for the multi-dimensional viscous Saint-Venant system with arbitrarily large initial data remained a challenging open problem. A breakthrough was recently achieved by Huang-Meng-Zhang \cite{Huang-Meng-Zhang}, who, under the assumption of radial symmetry, provided the first proof of global classical solutions for the two-dimensional initial-boundary value problem. At the same time,  Chen-Zhang-Zhu\cite{Chen-Zhang-Zhu} independently established the existence of global large solutions for the corresponding Cauchy problem. Significantly, Huang-Meng-Zhang \cite{Huang-Meng-Zhang} also derived global classical solutions for the broader class of isentropic compressible Navier-Stokes equations satisfying the BD entropy relation in both 2D and 3D—of which the shallow water equations are a specific case.
	
	Building upon the preceding analysis, this paper investigates the endpoint case of the BD entropy model and generalizes it to the non-isentropic system. For simplicity, we assume that $$\mu(\rho)=2\rho, \quad \lambda(\rho)=0.$$
	We consider the initial-boundary value problem \eqref{1.1} in a bounded domain $\Omega \subset \mathbb{R}^N$ ($N=2,3$), with the initial conditions
	\begin{equation} \label{1.2}
		(\rho, \rho \mathbf{u}, s)|_{t=0} = (\rho_0, \mathbf{m_0},s_0), \quad x \in \Omega,
	\end{equation}
	and the boundary condition
	\begin{equation} \label{1.3}
		\rho \mathbf{u}|_{\partial\Omega} = 0.
	\end{equation}
	
	For the radially symmetric setting, we assume $\Omega = B_{R}$ and focus on solutions of the form
	\begin{equation}\label{1.4}
		\rho(x,t) = \rho(r,t), \quad \mathbf{u}(x,t) = u(r,t)\frac{\mathbf{x}}{r}, \quad s(x,t)=s(r,t).
	\end{equation}
	In $N$-dimensional spherically symmetric coordinates, the system \eqref{1.1} takes the form
	\begin{equation}\label{3.3}
		\begin{cases}
			\rho_{t}+(\rho u)_{r}+\frac{N-1}{r}\rho u=0, \\
			\displaystyle \rho (u_t + u u_r) + P_r = 2 \partial_r(\rho u_r) + 2(N-1)\frac{\rho}{r}\left(u_r - \frac{u}{r}\right), \\
			\displaystyle s_t + u s_r = 0,
		\end{cases}
	\end{equation}
	with the initial data
	\begin{equation*}
		\rho(r, 0) = \rho_0(r), \quad \rho u(r, 0) = m_0(r), \quad s(r,0)=s_0(r),
	\end{equation*}
	and the boundary condition
	\begin{equation*}
		\rho u(0,t) = \rho u(R,t) = 0.
	\end{equation*}
	We begin by defining the global classical solution before formulating our theorem.
	\begin{definition}[Global classical solution]We say that a radially symmetric triple $(\rho, \mathbf{u}, s)$ with $\rho > 0$ is a global classical solution to the initial-boundary-value problem \eqref{1.1},\eqref{1.2} and \eqref{1.3} provided that, for any $ 0<\tau<T$,
		$$\begin{cases}
			\rho, s \in C([0,T];H^3(\Omega)), \quad \rho_t, s_t \in C([0,T];H^2(\Omega)),\\
			\rho_{tt}, s_{tt} \in L^{\infty}(0,T;L^2(\Omega)) \cap L^2(0,T;H^1(\Omega)),\\
			\mathbf{u} \in C([0,T];H_{0}^{1}(\Omega) \cap H^3(\Omega)) \cap L^2(0,T;H^4(\Omega)) \cap L^{\infty}(\tau,T;H^4(\Omega)),\\
			\mathbf{u}_t \in L^{\infty}(0,T;H_{0}^{1}(\Omega)) \cap L^2(0,T;H^2(\Omega)) \cap L^{\infty}(\tau,T;H^2(\Omega)),\\
			\mathbf{u}_{tt} \in L^2(0,T;L^2(\Omega)) \cap L^{\infty}(\tau,T;L^2(\Omega)) \cap L^2(\tau,T;H^1(\Omega)).
		\end{cases}$$
	\end{definition}
	Our main results are the following two theorems on the global existence of classical solutions with large initial data away from vacuum for the system with transported entropy.
	\begin{theorem}[Global classical solutions for $N=2$]\label{theorem 2.1}
		Assume that $$N=2, \quad \gamma>1,$$ 
		and the radially symmetric initial data $(\rho_0, \mathbf{u_0},s_0)$ satisfies 
		\begin{align*}
			& 0 < \underline{\rho} \le \rho_0 \le \bar{\rho}, \quad 0 < \underline{s} \le s_0 \le \bar{s}, \quad \frac{\partial_r s_0}{\rho_0 r} \in L^{\infty}(\Omega), \\
			& (\rho_0, \mathbf{u}_0, s_0) \in H^3(\Omega), \quad \mathbf{u}_0|_{\partial \Omega} = 0.
		\end{align*}
		Then, there exists a unique global radially symmetric classical solution to the initial-boundary value problem \eqref{1.1},\eqref{1.2} and \eqref{1.3} satisfying, for any $(x,t) \in \Omega \times[0,T]$, 
		\begin{equation}
			(C(T))^{-1} \le \rho(x,t) \le C(T), \quad 0 < \underline{s} \le s(x,t) \le \bar{s}.
		\end{equation}
		where the constant $C(T) > 0$ depends on the initial data and $T$.
	\end{theorem}
	\begin{theorem}[Global classical solutions for $N=3$] Assume that 
		$$N=3, \quad 1<\gamma<3,$$\label{theorem 2.2}
		and the spherically symmetric initial data $(\rho_0, \mathbf{u_0},s_0)$ satisfies 
		\begin{align*}
			&0<\underline{\rho }\le \rho _0\le \bar{\rho},\quad 0<\underline{s}\le s_0\le \bar{s},\quad \frac{\partial _rs_0}{\rho _0r^2}\in L^{\infty},
			\\
			&(\rho _0,\mathbf{u}_0,s_0)\in H^3(\Omega ),\quad \mathbf{u}_0|_{\partial \Omega}=0.
		\end{align*}
		Then, there exists a unique global spherically symmetric classical solution to the initial-boundary value problem \eqref{1.1},\eqref{1.2} and \eqref{1.3} satisfying, for any $(x,t) \in \Omega \times[0,T]$, 
		\begin{equation}
			(C(T))^{-1} \le \rho(x,t) \le C(T), \quad 0 < \underline{s} \le s(x,t) \le \bar{s}.
		\end{equation}
		where the constant $C(T) > 0$ depends on the initial data and $T$.
	\end{theorem}
	\begin{remark}
		When $s$ is constant, Theorem \ref{theorem 2.1} and Theorem \ref{theorem 2.2} can be regarded as an extension and improvement of the results by Huang-Meng-Zhang\cite{Huang-Meng-Zhang}. Here, we remove the restriction $\gamma \geq \frac{3}{2}$ in the two-dimensional case and extend the results to three dimensional case.
	\end{remark}
	The following corollary considers the case where $s$ is constant, namely the viscous shallow water model. This corresponds to the work of Huang-Meng-Zhang \cite{Huang-Meng-Zhang},
	where uniform upper and lower bounds on density and large-time behavior were obtained.
	\begin{corollary}[Uniform bounds on density and large-time behavior]\label{cor1.1}
		Assume that 
		$$N=2, \quad \gamma>1,$$
		or 
		$$N=3, \quad 1<\gamma<2,$$
		and the spherically symmetric initial data ($\rho_0,\mathbf{u}_0$) satisfies
		$$0<\underline{\rho }\le \rho _0\le \bar{\rho},\quad (\rho _0,\mathbf{u}_0)\in H^3(\Omega ),\quad \mathbf{u}_0|_{\partial \Omega}=0.
		$$
		Then, there exists a unique global spherically symmetric classical solution (see Huang-Meng-Zhang\cite{Huang-Meng-Zhang} Definition 2.6) to the initial-boundary value problem satisfying, for any $(x,t) \in \Omega \times[0,T]$, 
		\begin{equation*}
			C^{-1} \le \rho(x,t) \le C,
		\end{equation*}
		where the constant $C > 0$ depends on the initial data but is independent of $T$. Moreover, the following large-time behavior holds
		$$
		\lim_{t\to\infty} \|\rho(t) -\frac{1}{\left| \Omega \right|}\int_{\Omega}{\rho _0dx} \|_{C(\bar{\Omega})} = 0,
		$$
		and
		$$
		\lim_{t\to\infty} \left( \|\nabla\rho(t)\|_{L^2(\Omega)} + \|\nabla \mathbf{u}(t)\|_{L^2(\Omega)} + \|\nabla^2 \mathbf{u}(t)\|_{L^2(\Omega)} \right) = 0.
		$$
	\end{corollary}
	\begin{remark}
		The reason why the isentropic model admits uniform lower and upper bounds on density and large-time behavior is that uniform estimates for the density can be established. These estimates lead to the desired uniform density bounds and asymptotic behavior.
	\end{remark}
	We now introduce Lagrangian coordinates to simplify the following analysis.
	We define the coordinates transformation
	\begin{equation}
		y(r,t) = \int_0^r \rho(s,t)s^{N-1} \, ds, \quad \tau(r,t) = t.
	\end{equation}
	Without loss of generality, assume that $$\int_0^R \rho r^{N-1} \, dr = \int_0^R \rho_0 r^{N-1}\, dr = 1.$$ 
	Therefore, we translate the domain $[0, R] \times [0, T]$ into $[0, 1] \times [0, T]$, which satisfies
	\begin{equation}
		\frac{\partial y}{\partial r} = \rho r^{N-1}, \quad \frac{\partial y}{\partial t} = -\rho u r^{N-1}, \quad \frac{\partial \tau}{\partial r} = 0, \quad \frac{\partial \tau}{\partial t} = 1, \quad \frac{\partial r}{\partial \tau} = u.
	\end{equation}
	The system \eqref{3.3} transforms  into
	\begin{equation}\label{4.4}
		\begin{cases}
			\partial_{\tau}\rho + \rho^2 \partial_y(r^{N-1}u) = 0, \\
			\partial_{\tau}s=0,\\
			\displaystyle \frac{1}{r^{N-1}}\partial_{\tau}u + \partial_y P - \partial_y\left[ 2\rho^2 \partial_y(r^{N-1}u) \right] + \frac{2(N-1)}{r} u \partial_y \rho = 0.
		\end{cases}
	\end{equation}
	Let $w=u+2r^{N-1}\partial_y \rho$ be the effective velocity. Then the system for effective velocity becomes
	\begin{equation}\label{4.6}
		\begin{cases}
			\partial_{\tau}\rho + \rho^2 \partial_y(r^{N-1}u) = 0, \\
			\partial_{\tau}s=0,
			\\
			\partial_{\tau}w+r^{N-1}\partial_yP=0.
		\end{cases}
	\end{equation}
	We now outline the main strategy of the proof. We must address two major challenges: first, establishing the B-D entropy estimates for non-isentropic viscous shallow water waves; and second, improving upon the 2D indices of Huang-Meng-Zhang \cite{Huang-Meng-Zhang} and extending these results to the 3D case.\\
	\textbf{1. The BD entropy inequality for non-isentropic fluids.}\par
	In contrast to the isentropic case, where the pressure term provides a dissipation term when tested with the effective velocity $w=u+2r^{N-1}\partial_y \rho$, for non-isentropic flows, the entropy variation leads to the emergence of both a dissipation term and a source term. However, despite the lack of the standard BD entropy identity, the BD entropy inequality can still be derived by coupling the system with basic energy estimates. Therefore, we find a structure similar to the BD condition in the non-isentropic setting. In fact, in Lagrangian coordinates, we arrive at 
	$$\begin{aligned}
		&\frac{d}{d\tau} \left( \int_0^1 \frac{1}{2}w^2 \, dy + \int_0^1 \frac{\rho^{\gamma-1}}{\gamma-1} e^{s} \, dy \right) \\
		&\quad + \underbrace{ \int_0^1 2\gamma r^{2N-2} \rho^{\gamma-1} e^s (\partial_y \rho)^2 \, dy}_{\text{Dissipation term}}
		= -\underbrace{ \int_0^1 2 r^{2N-2} \rho^\gamma e^s (\partial_y s) (\partial_y \rho) \, dy}_{\text{Source term}}.
	\end{aligned}$$
	To handle the source term, we need to deal with the $\partial_y s$ at first. 
	In general, estimates for the derivatives of $s$ would require the $L^1 L^\infty$ norm of $\nabla \mathbf{u}$. Fortunately, the entropy is invariant along particle paths in the current setting, satisfying $s_{y\tau}=0$ in Lagrangian coordinates. Consequently, provided that the initial derivative $\partial_y s_0$ is bounded, $\partial_y s$ remains bounded throughout the evolution. For this reason, we impose an additional condition on the initial entropy in Eulerian coordinates, i.e. 
	$\frac{\partial _rs_0}{\rho _0r^{N-1}}\in L^{\infty}.$ In particular, the case $s=\text{const}$ satisfies this condition. Therefore, by using Young's inequality, we only need to deal with the term $\|\rho r^{N-1}\|_{L^\infty}^2$. By using weighted estimates near the origin
	and Sobolev's embedding in one-dimensional space
	$$
	\begin{aligned}
		\| \rho r^{N-1} \|_{L^{\infty}(0,R)}^{2}
		&\lesssim \left( \| \rho_y r^{N-1} \|_{L^1(0,1)}^{2} + \| \frac{1}{r} \|_{L^1(0,1)}^{2} \right) \\
		&\lesssim \left( \| \rho_y r^{N-1} \|_{L^1(0,1)}^{2} + \| \rho^{\frac{1}{2}}r^{\xi +\frac{1}{2}(N-2)} \|_{L^{\infty}} \right) \\
		&\lesssim \left( \| \rho_y r^{N-1} \|_{L^1(0,1)}^{2} + 1 \right) \\
		&\lesssim \left( \| w \|_{L^2(0,1)}^{2} + \| u \|_{L^2(0,1)}^{2} + 1 \right).
	\end{aligned}
	$$
	Consequently, we get the BD entropy inequality by combining with the basic energy identity and using Gronwall's inequality.\\
	\textbf{2. Global spherically symmetric classical solutions for the multi-dimensional case}\par 
	To derive the higher-order estimates and thereby prove the existence of global solutions, we first establish the lower-order estimates
	$$
	\begin{aligned}
		\mathop {\mathrm{sup}} \limits_{0\le t\le T}&\left( \| \mathbf{u}\| _{L^4} +\| \rho \| _{L^{\infty}}+\| \rho ^{-1}\| _{L^{\infty}}+\| \nabla \rho \| _{L^2}+\| \nabla \rho \| _{L^4} \right) 
		\\
		&+\int_0^T{\left( \| \nabla \rho \| _{L^2}^{2}+\| \nabla \rho \| _{L^4}^{4}+\|\nabla \mathbf{u}\| _{L^2}^{2} \right)}dt\le C(T).
	\end{aligned}
	$$
	The most crucial step here is to derive the upper and lower bounds for the density. In a certain sense, this serves as a blow-up criterion. We now proceed to outline the strategy for obtaining the upper and lower bounds of the density.\\
	\textbf{2.1. Upper bound for density}\par
	The BD entropy estimates initially yield a bound on $\| \nabla \rho^{1/2} \|_{L^2(\Omega)}$, but this is insufficient to determine the upper bound for the density. The crucial step lies in estimating $\| \nabla \rho^{\frac{1}{2n}} \|_{L^{2n}(\Omega)}$($n \ge 2$). We employ different methods to derive the upper bound of the density in the two-dimensional and three-dimensional cases. 
	\\
	\textbf{2.2. Lower bound for density}\par
	The reason we can extend the results of Huang-Meng-Zhang \cite{Huang-Meng-Zhang} for viscous shallow water equations to the range $\gamma > 1$ in 2D and $1 < \gamma < 3$ in 3D is that a lower bound for the density can be established in these cases. The essence of their method lies in deriving  $L^\infty$ bounds for both the effective velocity and the velocity field. Specifically, Huang-Meng-Zhang \cite{Huang-Meng-Zhang} first established the $L^\infty$ estimate for the effective velocity, followed by that of the velocity itself. A crucial step in their proof is the following critical estimate for spherically symmetric functions
	$$
	\begin{aligned}
		\| \sqrt{\rho}\mathbf{u} \| _{L^{\infty}(\Omega )} &\lesssim \| \sqrt{\rho}\nabla \mathbf{u} \| _{L^2(\Omega )} + \| \mathbf{u}\nabla \sqrt{\rho} \| _{L^2(\Omega )} \\
		&\lesssim \| \sqrt{\rho}\nabla \mathbf{u} \| _{L^2(\Omega )} + \| \nabla \rho ^{\frac{1}{4}} \| _{L^4(\Omega )} \| \rho ^{\frac{1}{4}}\mathbf{u} \| _{L^4(\Omega )}.
	\end{aligned}
	$$
	Consequently, this implies that estimating the effective velocity relies on controlling powers of the density, which leads to the restriction $\gamma \ge 3/2$ in the two-dimensional case. Furthermore, this dependency poses difficulties for extending the analysis to the three-dimensional setting.\par
	Our main observation is that, once the density upper bound is obtained, we can estimate $\|w\|_{L^\infty}$ and $\|u\|_{L^\infty}$ jointly. This is achieved by coupling two seemingly unrelated $L^{2n}$ estimates
	$$
	\begin{aligned}
		\frac{1}{2n}\frac{d}{d\tau} \|w\|_{L^{2n}}^{2n} 
		&\lesssim \int_0^1 \left| \rho^{\gamma -1}e^s u w^{2n-1} \right| \,dy + \int_0^1 \left| \rho^{\gamma}e^s r w^{2n-1}(\partial_y s) \right| \,dy \\
		&\lesssim \|w\|_{L^{2n}}^{2n} + \|u\|_{L^{2n}}^{2n} + 1, \\[1em]
		\frac{1}{2n}\frac{d}{d\tau} \|u\|_{L^{2n}}^{2n} 
		&\lesssim \int_0^1 \left| \partial_y P \right| r |u|^{2n-1} \,dy \\
		&\lesssim   \|w\|_{L^{2n}}^{2n} + \|u\|_{L^{2n}}^{2n}+1 .
	\end{aligned}
	$$
	Combining these estimates and applying Gronwall's inequality, we establish the $L^\infty$ bounds for the velocity and the effective velocity as $n \to \infty$. One can observe an intrinsic connection between $u$ and $w$. Moreover, since this derivation does not rely on the assumption of spherical symmetry, it may provide insights for the study of general domains.
	\section{Preliminary}
	The local existence theory for classical solutions stated below can be established using standard methods similar to those in \cite{Li-Pan-Zhu,Zhang-Zhao,Zhang}, so we omit the proof.
	\begin{lemma}\label{lemma2.1}
		Assume that the spherically symmetric initial data $(\rho_0, \mathbf{u}_0,s_0)$ and $\gamma$ satisfy the conditions of Theorem \ref{theorem 2.1} and Theorem \ref{theorem 2.2}. Then there exists a small time $T_0 > 0$ such that the initial-boundary-value problem \eqref{1.1},\eqref{1.2} and \eqref{1.3} has a unique symmetric classical solution $(\rho, \mathbf{u}, s)$ on $\Omega \times [0, T_0]$ with $\rho,s > 0$, and for every $0 < \tau < T_0$, satisfying
		$$\begin{cases}
			\rho, s \in C([0,T_0];H^3(\Omega)), \quad \rho_t, s_t \in C([0,T_0];H^2(\Omega)),\\
			\rho_{tt}, s_{tt} \in L^{\infty}(0,T_0;L^2(\Omega)) \cap L^2(0,T_0;H^1(\Omega)),\\
			\mathbf{u} \in C([0,T_0];H_{0}^{1}(\Omega) \cap H^3(\Omega)) \cap L^2(0,T_0;H^4(\Omega)) \cap L^{\infty}(\tau,T_0;H^4(\Omega)),\\
			\mathbf{u}_t \in L^{\infty}(0,T_0;H_{0}^{1}(\Omega)) \cap L^2(0,T_0;H^2(\Omega)) \cap L^{\infty}(\tau,T_0;H^2(\Omega)),\\
			\mathbf{u}_{tt} \in L^2(0,T_0;L^2(\Omega)) \cap L^{\infty}(\tau,T_0;L^2(\Omega)) \cap L^2(\tau,T_0;H^1(\Omega)).
		\end{cases}$$
	\end{lemma}
	We now present two lemmas that are crucial for the subsequent estimates. These results establish the relationship between the weighted estimates near the origin and $\|\nabla \rho^{1/2}\|_{L^2(\Omega)}$.
	\begin{lemma}\label{lemma 2D}
		In the two-dimensional radially symmetric domain $B_R$, for any $0 < \xi \ll 1$, we have
		\begin{equation} \label{5.13}
			\mathop {\mathrm{sup}} \limits_{0\le t\le T}\| \rho ^{\frac{1}{2}}r^{\xi} \| _{L^{\infty}(0,R)}\le  C(\xi) \left( 1 + \|\nabla \rho^{\frac{1}{2}}\|_{L^2(\Omega)} \right).
		\end{equation}
		where $C(\xi)$ is a positive constant independent of $T$ and depending on $R$ and $\| \rho \|_{L^1(\Omega)}$.
	\end{lemma}
	\begin{proof}
		Using Sobolev's embedding inequality in one dimension and Cauchy-Schwarz inequality,
		\begin{equation}
			\begin{aligned}
				\|\rho^{\frac{1}{2}}r^{\xi}\|_{L^{\infty}(0,R)} 
				&\le C\int_0^R \rho^{\frac{1}{2}}r^{\xi}\,dr + C\int_0^R |\partial_r \rho^{\frac{1}{2}}|r^{\xi}\,dr + C\int_0^R \rho^{\frac{1}{2}}r^{\xi -1}\,dr \\
				&\le C\int_{\Omega} \rho^{\frac{1}{2}}r^{\xi -1}\,dx + C\int_{\Omega} |\nabla \rho^{\frac{1}{2}}|r^{\xi -1}\,dx + C\int_{\Omega} \rho^{\frac{1}{2}}r^{\xi -2}\,dx \\
				&\le C \left( \int_{\Omega} |\nabla \rho^{\frac{1}{2}}|^2\,dx \right)^{\frac{1}{2}} \left( \int_{\Omega} r^{2\xi -2}\,dx \right)^{\frac{1}{2}} \\
				&\quad + C \left( \int_{\Omega} \rho^{\frac{2}{\xi}}\,dx \right)^{\frac{\xi}{4}} \left( \int_{\Omega} r^{\frac{4\xi -8}{4-\xi}}\,dx \right)^{\frac{4-\xi}{4}} \\
				&\le C(\xi) \left( \int_{\Omega} |\nabla \rho^{\frac{1}{2}}|^2\,dx \right)^{\frac{1}{2}} + C(\xi) \|\rho^{\frac{1}{2}}\|_{L^{\frac{4}{\xi}}(\Omega)} \\
				&\le C(\xi) \|\nabla \rho^{\frac{1}{2}}\|_{L^2(\Omega)} + C(\xi) \left( \|\nabla \rho^{\frac{1}{2}}\|_{L^2(\Omega)} + \|\rho^{\frac{1}{2}}\|_{L^2(\Omega)} \right) \\
				&= C(\xi) \left( 1 + \|\nabla \rho^{\frac{1}{2}}\|_{L^2(\Omega)} \right).
			\end{aligned}
		\end{equation}
	\end{proof}
	\begin{lemma}\label{lemma 3D}
		In the three-dimensional radially spherically symmetric domain $B_R$, for any $0 < \xi \ll 1$, we have
		\begin{equation} \label{5.14}
			\mathop {\mathrm{sup}} \limits_{0\le t\le T}\| \rho ^{\frac{1}{2}}r^{\frac{1}{2}+\xi}\| _{L^{\infty}(0,R)}\le  C(\xi) \left( 1 + \|\nabla \rho^{\frac{1}{2}}\|_{L^2(\Omega)} \right).
		\end{equation}
		where $C(\xi)$ is a positive constant independent of $T$ and depending on $R$ and $\| \rho \|_{L^1(\Omega)}$.
	\end{lemma}
	\begin{proof}
		Using Sobolev's embedding inequality in one-dimension and Cauchy-Schwarz inequality,
		\begin{equation}\label{5.58}
			\begin{aligned}
				\| \rho ^{\frac{1}{2}}r^{\frac{1}{2}+\xi}\| _{L^{\infty}(0,R)}&\le C\int_0^R{\rho ^{\frac{1}{2}}}r^{\frac{1}{2}+\xi}\,dr+C\int_0^R{|}\partial _r\rho ^{\frac{1}{2}}|r^{\frac{1}{2}+\xi}\,dr+C\int_0^R{\rho ^{\frac{1}{2}}}r^{\xi -\frac{1}{2}}\,dr\\
				&\le C\int_{\Omega}{\rho ^{\frac{1}{2}}}r^{\xi -\frac{3}{2}}\,dx+C\int_{\Omega}{|}\nabla \rho ^{\frac{1}{2}}|r^{\xi -\frac{3}{2}}\,dx+C\int_{\Omega}{\rho ^{\frac{1}{2}}}r^{\xi -\frac{5}{2}}\,dx\\
				&\le C\left( \int_{\Omega}{|}\nabla \rho ^{\frac{1}{2}}|^2\,dx \right) ^{\frac{1}{2}}\left( \int_{\Omega}{r^{2\xi -3}}\,dx \right) ^{\frac{1}{2}}\\
				&\quad +C\left( \int_{\Omega}{\rho ^3}\,dx \right) ^{\frac{1}{6}}\left( \int_{\Omega}{r^{\frac{6}{5}\xi -3}}\,dx \right) ^{\frac{5}{6}}\\
				&\le C(\xi )\left( \int_{\Omega}{|}\nabla \rho ^{\frac{1}{2}}|^2\,dx \right) ^{\frac{1}{2}}+\left( \left( \int_{\Omega}{|}\nabla \rho ^{\frac{1}{2}}|^2\,dx \right) ^{\frac{1}{2}}+\| \rho ^{\frac{1}{2}} \| _{L^2(\Omega)} \right) \left( \int_{\Omega}{r^{\frac{6}{5}\xi -3}}\,dx \right) ^{\frac{5}{6}}\\
				&\le C(\xi )(\| \nabla \rho ^{\frac{1}{2}}\| _{L^2(\Omega)}+1).
			\end{aligned}
		\end{equation}
	\end{proof}
	
	\section{A priori estimates: Lower-Order Estimates for $N=2$}
	In this section, we establish lower-order estimates for the two-dimensional non-isentropic viscous shallow water equations. We begin by proving the BD entropy inequality.
	\subsection{$L^2$ estimates for velocity and effective velocity}
	\begin{proposition} \label{prop3.1}
		There exists a constant $C(T)> 0$, dependent of $T$, such that
		\begin{equation} \label{10}
			\sup_{0 \le \tau \le T} \int_0^1 (u^2 +\rho ^{\gamma -1}) \, dy + \int_0^T \int_0^1 \left( \frac{u^2}{r^2} + \rho^2 (\partial_y u)^2 r^2 \right) \, dy d\tau \le C,
		\end{equation}
		\begin{equation} \label{11}
			\sup_{0 \le \tau \le T} \int_0^1 (u + 2r\partial_y \rho)^2 \, dy + \int_0^T \int_0^1 \rho^{\gamma-1} (r\partial_y \rho)^2 \, dy d\tau \le C(T).
		\end{equation}
	\end{proposition}
	
	\begin{proof}
		Multiplying the second equation of \eqref{4.4} by $ru$, we get basic energy estimate
		\begin{equation}\label{5.8}
			\frac{d}{d\tau} \int_0^1 \left( \frac{1}{2}u^2 + \frac{\rho^{\gamma -1}}{\gamma -1}e^s \right) dy + \int_0^1 \left( 2\rho^2 r^2 (\partial_y u)^2 + \frac{2u^2}{r^2} \right) dy = 0.
		\end{equation}
		Therefore, we prove \eqref{10}.
		\par
		Now multiplying the third equation of \eqref{4.6} by $w=u+2r\partial_y \rho$, we get
		$$
		\begin{aligned}
			\frac{1}{2}\frac{d}{d\tau}\int_0^1{w^2}\,dy&=-\int_0^1{\partial _yP}(ru)\,dy-2\int_0^1{\partial _yP}(r^2\partial _y\rho )\,dy
			\\
			&=\int_0^1{P}\partial _y(ru)\,dy-2\int_0^1{\partial _yP}\left( r^2\partial _y\rho \right) \,dy\triangleq H_1+H_2.
		\end{aligned}
		$$
		Using \eqref{4.4}$_1$, $H_1$ is equivalent to
		$$H_1 = \int_0^1 P \left( -\frac{\partial_\tau \rho}{\rho^2} \right) dy \\= -\frac{d}{d\tau} \int_0^1 \frac{\rho^{\gamma-1}}{\gamma-1} e^{s} \, dy,$$
		and
		$$
		\begin{aligned}
			H_2 &= - \int_0^1 \left[ \gamma \rho^{\gamma-1} e^s (\partial_y \rho) + \rho^\gamma e^s (\partial_y s) \right] \cdot (2r^2 \partial_y \rho) \, dy \\
			&= {- \int_0^1 2\gamma r^2 \rho^{\gamma-1} e^s (\partial_y \rho)^2 \, dy} - \int_0^1{ 2 r^2 \rho^\gamma e^s (\partial_y s) (\partial_y \rho) \, dy}.
		\end{aligned}$$
		So we arrive at 
		\begin{equation}\label{55.6}
			\frac{d}{d\tau} \left( \int_0^1 \frac{1}{2}w^2 \, dy + \int_0^1 \frac{\rho^{\gamma-1}}{\gamma-1} e^{s} \, dy \right) + \int_0^1 2\gamma r^2 \rho^{\gamma-1} e^s (\partial_y \rho)^2 \, dy = - \int_0^1 2 r^2 \rho^\gamma e^s (\partial_y s) (\partial_y \rho) \, dy.
		\end{equation}
		Using Young's inequality, the source term on the RHS can be estimated as follows$$\begin{aligned}
			- \int_0^1 2 r^2 \rho^\gamma e^s (\partial_y s)(\partial_y \rho) \, dy 
			&\le \int_0^1 \gamma r^2 \rho^{\gamma-1} e^s (\partial_y \rho)^2 \, dy + C \int_0^1 r^2 \rho^{\gamma+1} e^s (\partial_y s)^2 \, dy \\
			&\le \int_0^1 \gamma r^2 \rho^{\gamma-1} e^s (\partial_y \rho)^2 \, dy + C \|\rho r\|_{L^\infty}^2 \int_0^1 \rho^{\gamma-1} \, dy.
		\end{aligned}$$
		Here we use $e^{\underline{s}} \le e^s\le e^{\bar{s}}\le C,\,\,\left\| \partial _ys \right\| _{L^{\infty}}\le \left\| \partial _ys_0 \right\| _{L^{\infty}}\le C$. Substituting this back into \eqref{55.6}, and absorbing the first term into the dissipation on the LHS, we obtain
		\begin{equation}\label{5.12}
			\frac{d}{d\tau} \left( \int_0^1 \frac{1}{2}w^2 \, dy + \int_0^1 \frac{\rho^{\gamma-1}}{\gamma-1} e^s\, dy \right)+ \int_0^1 \gamma r^2 \rho^{\gamma-1} e^s (\partial_y \rho)^2 \, dy \le C \|\rho r\|_{L^\infty}^2,
		\end{equation}
		The key step in the following process is to estimate $\|\rho r\|_{L^\infty}$. By the fundamental theorem of calculus, we have
		\begin{equation}
			\left( \rho r \right) \left( y \right) =\left( \rho r \right) \left( y \right) -\left( \rho r \right) \left( 0 \right) \le \int_0^1{\left| \left( \rho r \right) _y \right|dy}\le \int_0^1{\left| \rho _yr \right|+\left| r_y\rho \right|dy},
		\end{equation}
		so we derive
		$$\begin{aligned}
			\|\rho r\| _{L^{\infty}}^{2}&\le C\left( \| \rho _yr\| _{L^1(0,1)}+\left\| \rho r_y \right\| _{L^1(0,1)} \right) ^2\\
			&=C\left( \| \rho _yr\| _{L^1(0,1)}^{2}+\| \frac{1}{r} \| _{L^1(0,1)}^{2} \right),
		\end{aligned}$$
		where we used the relation $\partial_y r = \frac{1}{\rho r}$. Now we estimate each term on the RHS. First, by Hölder's inequality, we get
		\begin{equation}
			\|\rho_y r\|_{L^1}^2 = \left( \int_0^1 |\rho_y r| \, dy \right)^2 \le \left( \int_0^1 |\rho_y r|^2 \, dy \right) \left( \int_0^1 1 \, dy \right) = \|\rho_y r\|_{L^2}^2.
		\end{equation}
		Second, for the term $\| \frac{1}{r}\|_{L^1}$, using the coordinate transformation $dy = \rho r dr$, we have
		\begin{equation}\label{33.88}
			\begin{aligned}
				\| \frac{1}{r}\| _{L^1}^{2}&=\left( \int{\frac{1}{r}}\,dy \right) ^2=\left( \int{\rho}\,dr \right) ^2\\
				&=\left( \int{\rho ^{\frac{1}{2}}}r^{\xi}\cdot \rho ^{\frac{1}{2}}r^{-\xi}\,dr \right) ^2\\
				&\le \| \rho ^{\frac{1}{2}}r^{\xi}\| _{L^{\infty}(0,R)}^{2}\left( \int{\rho ^{\frac{1}{2}}}r^{\frac{1}{2\gamma}}r^{-\frac{1}{2\gamma}-\xi}\,dr \right) ^2\\
				&\le \| \rho ^{\frac{1}{2}}r^{\xi}\| _{L^{\infty}(0,R)}^{2}(\int{\rho ^{\gamma}}rdr)^{\frac{1}{\gamma}}(\int{r^{-\frac{1}{2\gamma -1}-\frac{2\gamma}{2\gamma -1}\xi}}dr)^{\frac{2\gamma-1}{\gamma}}\\
				&\le C(\xi )\| \rho ^{\frac{1}{2}}r^{\xi}\| _{L^{\infty}(0,R)}^{2},\\
			\end{aligned}    
		\end{equation}
		where $\xi$ is small enough s.t. 
		$$\xi <1-\frac{1}{\gamma}.$$
		By Lemma \ref{lemma 2D}, in two-dimensional space, the weighted estimate at the origin is
		$$
		\| \rho ^{\frac{1}{2}}r^{\xi}\| _{L^{\infty}(0,R)}\le C(\xi )(\| \nabla \rho ^{\frac{1}{2}}\| _{L^2}+1)\\=C(\xi )\left( 1+\| \rho _yr\|_{L^2(0,1)} \right) .\\
		$$
		Combining \eqref{5.12}-\eqref{33.88}, we arrive at 
		\begin{equation}\label{5.14}
			\frac{d}{d\tau}\left( \int_0^1{\frac{1}{2}}w^2\,dy+\int_0^1{\frac{\rho^{\gamma -1}}{\gamma -1}e^s}\,dy \right) +\int_0^1{2}\gamma r^2\rho ^{\gamma -1}e^s(\partial _y\rho )^2\,dy
			\\
			\le C(\xi )\left( 1+\| \rho _yr\| _{L^2(0,1)}^{2} \right) .
		\end{equation}
		Combining \eqref{5.8} and \eqref{5.14}, we obtain
		\begin{equation}
			\begin{split}
				\frac{d}{d\tau}&\int_0^1{\left( \frac{1}{2}u^2+\frac{1}{2}w^2+\frac{\rho^{\gamma -1}}{\gamma -1}e^{s} \right)}dy+\int_0^1{\left( 2\rho ^2r^2(\partial _yu)^2+\frac{2u^2}{r^2} \right)+{2}\gamma r^2\rho ^{\gamma -1}e^s(\partial _y\rho )^2}dy
				\\
				&\le C(\xi )\left( 1+\| \rho _yr\| _{L^2(0,1)}^{2} \right) .
				\\
				&\le C(\xi )\left( 1+\left\| u \right\| _{L^2(0,1)}^{2}+\left\| w \right\| _{L^2(0,1)}^{2} \right).
			\end{split}
		\end{equation}
		Let the total energy functional be defined as
		\begin{equation}
			E(\tau) = \int_0^1 \left( \frac{1}{2}u^2 + \frac{1}{2}w^2 + \frac{\rho^{\gamma -1}}{\gamma -1}e^{\gamma} \right) dy.
		\end{equation}
		Using the definition of the effective velocity $w - u = 2r \partial_y \rho$, we have the relation
		\begin{equation}
			\|\rho_y r\|_{L^2(0,1)}^2 = \| \frac{w-u}{2} \|_{L^2(0,1)}^2 \le \frac{1}{2} \left( \|u\|_{L^2(0,1)}^2 + \|w\|_{L^2(0,1)}^2 \right) \le E(\tau).
		\end{equation}
		Substituting this into the derived energy estimate, we obtain the following differential inequality
		\begin{equation}
			\begin{aligned}
				\frac{d}{d\tau} E(\tau) 
				&\le C(\xi) \left( 1 + \|\rho_y r\|_{L^2}^2 \right) \\
				&\le C(\xi) \left( 1 + \|u\|_{L^2}^2 + \|w\|_{L^2}^2 \right) \\
				&\le C(\xi) \left( 1 + 2E(\tau) \right).
			\end{aligned}
		\end{equation}
		Applying Gronwall's inequality yields the uniform bound for the energy
		\begin{equation}
			E(t) \le \left( E(0) + C(\xi)t \right) e^{2C(\xi)t}, \quad \text{for any } t \in [0, T].
		\end{equation}
		This implies \eqref{11}. Consequently, we also obtain
		\begin{equation}\label{5.20}
			\|\rho^{\frac{1}{2}}r^{\xi}\|_{L^{\infty}(0,R)} \le C(T),
		\end{equation}
		and 
		\begin{equation}
			\|\rho_y r\|_{L^2(0,1)} \leq C(T).
		\end{equation}
	\end{proof}
	\subsection{$L^4$ estimates for velocity and effective velocity}
	The following property is used to establish the boundedness of $\| \nabla \rho^{1/4} \|_{L^4(\Omega)}$, which is crucial for proving the upper bound of the density.
	\begin{proposition} \label{6.1}
		There exists a constant $C > 0$, dependent of $T$, such that
		\begin{equation} \label{eq:13}
			\sup_{0 \le \tau \le T} \int_0^1 u^4 \, dy + \int_0^T \int_0^1 \left( \frac{u^4}{r^2} + \rho^2r^2 (\partial_y u)^2 u^2\right) \, dy d\tau \le C,
		\end{equation}
		\begin{equation} \label{6.2}
			\sup_{0 \le \tau \le T} \int_0^1 (u + 2r\partial_y \rho)^4 \, dy + \int_0^T \int_0^1 \rho^{\gamma-1} (r\partial_y \rho)^4 \, dy d\tau \le C.
		\end{equation}
	\end{proposition}
	
	\begin{proof}
		Multiplying the momentum equation by $r u^3$ and integrating over the spatial domain $[0,1]$, we obtain the following equality
		\begin{equation}\label{5.21}
			\frac{1}{4}\frac{d}{d\tau}\int_0^1{u^4}\,dy=\int_0^1{\left( -\partial _yP+\partial _y\left( 2\rho ^2\partial _y(ru) \right) -\frac{2}{r}\partial _y\rho u \right)}ru^3\,dy\triangleq J_1+J_2+J_3.
		\end{equation}
		For the pressure term $J_1$, we integrate by parts and use $P = \rho^\gamma e^s$
		\begin{equation}\label{5.22}
			\begin{aligned}
				J_1 &= -\int_0^1 (\partial_y P) r u^3 \, dy = \int_0^1 P \partial_y(r u^3) \, dy \\
				&= \int_0^1 \rho^\gamma e^s \left( \frac{u^3}{\rho r} + 3 r u^2 \partial_y u \right) \, dy \\
				&=\int_0^1 \frac{\rho^{\gamma-1} e^s u^3}{r} \, dy+ \int_0^1 3 \rho^\gamma e^s r u^2 \partial_y u \, dy.
			\end{aligned}
		\end{equation}
		Using integration by parts and the relation $\partial_y r = \frac{1}{\rho r}$, the viscous term can be expanded as
		\begin{equation}\label{5.23}
			\begin{aligned}
				J_2 &=\int_0^1 \partial_y \left( 2\rho^2 \partial_y(ru) \right) r u^3 \, dy \\
				&= -\int_0^1 2\rho^2 \partial_y(ru) \partial_y(r u^3) \, dy \\
				&= -\int_0^1 2\rho^2 \left( \frac{u}{\rho r} + r \partial_y u \right) \left( \frac{u^3}{\rho r} + 3 r u^2 \partial_y u \right) \, dy \\
				&= -\int_0^1 \left( \frac{2 u^4}{r^2} + 8 \rho u^3 \partial_y u + 6 \rho^2 r^2 u^2 (\partial_y u)^2 \right) \, dy.
			\end{aligned}
		\end{equation}
		The source term, after integration by parts on $\partial_y \rho$, becomes
		\begin{equation}\label{5.24}
			\begin{aligned}
				J_3 &=- \int_0^1 \frac{2}{r} (\partial_y \rho) u \cdot r u^3 \, dy =- \int_0^1 2 u^4 \partial_y \rho \, dy \\
				&=  \int_0^1 \rho \partial_y (2 u^4) \, dy = \int_0^1 8 \rho u^3 \partial_y u \, dy.
			\end{aligned}
		\end{equation}
		Substituting \eqref{5.22}-\eqref{5.24} into \eqref{5.21}, we have
		$$\frac{1}{4} \frac{d}{d\tau} \|u\|_{L^4}^4 + \int_0^1 \left( \frac{2 u^4}{r^2} + 6 \rho^2 r^2 u^2 (\partial_y u)^2 \right) \, dy 
		= \int_0^1 \frac{\rho^{\gamma-1} e^s u^3}{r} \, dy + \int_0^1 3 \rho^\gamma e^s r u^2 \partial_y u \, dy.$$
		The first term on the RHS can be estimated as 
		\begin{equation}
			\int_0^1{\frac{\rho ^{\gamma -1}e^su^3}{r}}dy\le \frac{1}{2}\int_0^1{\frac{u^4}{r^2}dy}+C\int_0^1{\rho ^{2\gamma -2}u^2}dy.
		\end{equation}
		By applying Young's inequality to the last term on the RHS
		$$\left| \int_0^1 3 \rho^\gamma e^s r u^2 \partial_y u \, dy \right| \le \frac{1}{2} \int_0^1 6 \rho^2 r^2 u^2 (\partial_y u)^2 \, dy + C \int_0^1 \rho^{2\gamma-2} u^2 \, dy.$$
		Absorbing the dissipation term into the LHS, we finally obtain the desired $L^4$ 
		estimate
		\begin{equation}
			\begin{split}
				\frac{d}{d\tau}\| u \| _{L^4}^{4}+\int_0^1{\left( \frac{u^4}{r^2}+\rho ^2r^2u^2(\partial _yu)^2 \right)}\,dy &\le C\int_0^1{\rho ^{2\gamma -2}u^2}\,dy
				\\
				&\le C \| \rho ^{\frac{1}{2}}r^{\xi} \| _{L^{\infty}}^{4\gamma -4}\int_0^1{\frac{u^2}{r^2}}\,dy\le C\left( T \right),
			\end{split}
		\end{equation}
		where $\xi$ is small enough s.t.
		\begin{equation}
			\xi <\frac{1}{2\left( \gamma -1 \right)}.
		\end{equation}
		Next, multiplying $\eqref{4.4}_2$ with $(u + 2r\partial_y \rho)^3$ and integrating over the spatial domain $[0,1]$, we obtain
		\begin{equation}\label{33.225}
			\frac{1}{4} \frac{d}{d\tau} \int_0^1 w^4 \, dy + \int_0^1 r w^3 \partial_y P \, dy = 0.
		\end{equation}
		Since $\partial_y P = \gamma \rho^{\gamma-1} e^s \partial_y \rho + \rho^\gamma e^s \partial_y s$, we arrive at
		\begin{equation}
			\int_0^1 r w^3 \partial_y P \, dy = \int_0^1 \gamma \rho^{\gamma-1} e^s r w^3 (\partial_y \rho) \, dy + \int_0^1 \rho^\gamma e^s r w^3 (\partial_y s) \, dy.
		\end{equation}
		Substituting the relation $\partial_y \rho = \frac{w - u}{2r}$ into the first term, we get
		\begin{equation}\label{33.227}
			\begin{aligned}
				\int_0^1 \gamma \rho^{\gamma-1} e^s r w^3 (\partial_y \rho) \, dy 
				&= \int_0^1 \gamma \rho^{\gamma-1} e^s r w^3 \left( \frac{w - u}{2r} \right) \, dy \\
				&= \frac{\gamma}{2} \int_0^1 \rho^{\gamma-1} e^s w^4 \, dy - \frac{\gamma}{2} \int_0^1 \rho^{\gamma-1} e^s u w^3 \, dy.
			\end{aligned}
		\end{equation}
		Combining \eqref{33.225}-\eqref{33.227} we derive the energy equality
		\begin{equation}\label{6.5}
			\frac{1}{4} \frac{d}{d\tau} \|w\|_{L^4}^4 + \frac{\gamma}{2} \int_0^1 \rho^{\gamma-1} e^s w^4 \, dy 
			= \frac{\gamma}{2} \int_0^1 \rho^{\gamma-1} e^s u w^3 \, dy - \int_0^1 \rho^\gamma e^s r w^3 (\partial_y s) \, dy.
		\end{equation}
		Applying Young's inequality to the first term on the RHS of \eqref{6.5}
		\begin{equation}\label{6.6}
			\left| \frac{\gamma}{2} \int_0^1 \rho^{\gamma-1} e^s u w^3 \, dy \right| 
			\le \frac{\gamma}{4} \int_0^1 \rho^{\gamma-1} e^s w^4 \, dy + C \int_0^1 \rho^{\gamma-1} e^s u^4 \, dy,    
		\end{equation}
		where $\int_0^1{\rho ^{\gamma -1}}e^su^4$ can be bounded by
		\begin{equation}
			\begin{split}
				\int_0^1{\rho ^{\gamma -1}}e^su^4\,dy&\le C\int_0^1{\rho ^{\gamma -1}}u^4\,dy
				\\
				&\le C\left\| \rho ^{\frac{1}{2}}r^{\xi} \right\| _{L^{\infty}\left( 0,R \right)}^{2\gamma -2}\int_0^1{\frac{u^4}{r^2}}\,dy
				\\
				&\le C\left( T \right)\int_0^1{\frac{u^4}{r^2}}\,dy .
			\end{split}
		\end{equation}
		For the second source term in \eqref{6.5}, using the radially weighted estimate \eqref{5.20}, we obtain
		\begin{equation}\label{6.8}
			\begin{split}
				\int_0^1{\rho ^{\gamma}}e^srw^3(\partial _ys)\,dy&\le \frac{\gamma}{8}\int_0^1{\rho ^{\gamma -1}e^sw^4dy}+C\int_0^1{\rho ^{\gamma +3}r^4dy}
				\\
				&\le \frac{\gamma}{8}\int_0^1{\rho ^{\gamma -1}e^sw^4dy}+C\left\| \rho ^{\frac{1}{2}}r^{\xi} \right\| _{L^{\infty}\left( 0,R \right)}^{2\gamma +6}\int_0^1{r^{4-\xi \left( 2\gamma +6 \right)}dy}
				\\
				&\le \frac{\gamma}{8}\int_0^1{\rho ^{\gamma -1}e^sw^4dy}+C\left( T,R \right),
			\end{split}
		\end{equation}
		where $\xi$ is small enough s.t.
		$$ \xi \le \frac{2}{\gamma+3}.$$
		Inserting \eqref{6.6}--\eqref{6.8} into \eqref{6.5}, we derive
		\begin{equation}
			\sup_{0\le \tau \le T} \|w\|_{{L^4}(0,1)}^4 + \int_0^T \int_0^1 \rho^{\gamma-1} w^4 \, dy d\tau \le C(T).
		\end{equation}
	\end{proof}
	\subsection{Upper bound estimates for the density}
	The $L^4$ estimates for the velocity field and the effective velocity allow us to obtain the upper bound of the density.
	\begin{proposition}\label{7.1}
		There exists a constant $C > 0$, dependent of $T$, such that
		\begin{equation}
			\sup_{0 \le t \le T} \| \rho \|_{L^\infty(0,R)} \le C(T).
		\end{equation}
	\end{proposition}
	\begin{proof}
		Since from \eqref{6.1} and \eqref{6.2}, we arrive at 
		\begin{equation}
			\mathop {\mathrm{sup}} \limits_{0\le t\le T}\left\| r\partial _y\rho \right\| _{{L^4}(0,1)}^{4}\le \mathop {\mathrm{sup}} \limits_{0\le t\le T}\left\| w-u \right\| _{{L^4}(0,1)}^{4}\le \mathop {\mathrm{sup}} \limits_{0\le t\le T}\left\| w \right\| _{{L^4}(0,1)}^{4}+\mathop {\mathrm{sup}} \limits_{0\le t\le T}\left\| u \right\| _{{L^4}(0,1)}^{4}\le C\left( T \right) .
		\end{equation}
		In Euler coordinates, this is equivalent to 
		$$\mathop {\mathrm{sup}} \limits_{0\le t\le T}\int_0^R{\left| \partial _r\rho ^{\frac{1}{4}} \right|^4}r\,dr\le C(T).$$
		By virtue of the two-dimensional Sobolev embedding, we obtain
		\begin{equation}
			\mathop {\mathrm{sup}} \limits_{0\le t\le T}\| \rho ^{\frac{1}{4}} \| _{L^{\infty}}\le \| \rho ^{\frac{1}{4}} \| _{L^{\infty}L^4}+\| \nabla \rho ^{\frac{1}{4}} \| _{L^{\infty}L^4}.
		\end{equation}
		Therefore, we get the upper bound of the density.
	\end{proof}
	\subsection{$L^{\infty}$ estimates for velocity and effective velocity}
	We now deduce the $L^\infty$ estimates for the velocity and effective velocity fields by combining their $L^{2n}$ estimates. This is a crucial step for establishing the density lower bound.
	\begin{proposition}\label{pro4.4}
		There exists a constant $C(T) > 0$ such that
		\begin{equation}\label{88.1}
			\sup_{0\le \tau \le T} \|w(\cdot, \tau)\|_{L^\infty(0,1)} \le C(T).
		\end{equation}
		and
		\begin{equation}\label{88.2}
			\sup_{0\le \tau \le T} \|u(\cdot, \tau)\|_{L^\infty(0,1)} \le C(T).
		\end{equation}
	\end{proposition}
	\begin{proof}
		Fix $n \in \mathbb{N}$, multiplying the effective momentum equation $\partial_\tau w + r\partial_y P = 0$ by $w^{2n-1}$ and integrating over $[0,1]$
		$$\frac{1}{2n} \frac{d}{d\tau} \|w\|_{L^{2n}}^{2n} + \int_0^1 \gamma \rho^{\gamma-1} e^s r w^{2n-1} (\partial_y \rho) \, dy = - \int_0^1 \rho^\gamma e^s r w^{2n-1} (\partial_y s) \, dy.$$
		Substituting the relation $\partial_y \rho = \frac{w-u}{2r}$ into the second term on the LHS, we arrive at
		\begin{equation}\label{8.1}
			\begin{aligned}
				\frac{1}{2n} \frac{d}{d\tau} \|w\|_{L^{2n}}^{2n} + \frac{\gamma}{2} \int_0^1 \rho^{\gamma-1} e^s w^{2n} \, dy 
				&= \frac{\gamma}{2} \int_0^1 \rho^{\gamma-1} e^s u w^{2n-1} \, dy \\
				&\quad - \int_0^1 \rho^\gamma e^s r w^{2n-1} (\partial_y s) \, dy\triangleq G_1+G_2.
			\end{aligned}
		\end{equation}
		$G_1$ can be estimated as follows
		\begin{equation}\label{8.2}
			\begin{aligned}
				\left| \int_0^1{\rho ^{\gamma -1}}e^suw^{2n-1}\,dy \right|&\le C\left( T \right) \| e^s\| _{L^{\infty}}\left\| u \right\| _{L^{2n}}\| w\| _{L^{2n}}^{2n-1}\\
				&\le C(T)\left( \left\| u \right\| _{L^{2n}}^{2n}+\| w\| _{L^{2n}}^{2n} \right) .\\
			\end{aligned}
		\end{equation}
		Since $s_{y\tau}=0$, we have $\|\partial_y s\|_{L^\infty} = \|\partial_y s_0\|_{L^\infty} \le C$.
		The term $G_2$ can be bounded as follows
		\begin{equation}\label{8.3}
			\left| \int_0^1 \rho^\gamma e^s r w^{2n-1} (\partial_y s) \, dy \right| \le C \|\rho\|_{L^\infty}^\gamma \|w\|_{L^{2n}}^{2n-1} \le C(T) (1 + \|w\|_{L^{2n}}^{2n}).
		\end{equation}
		Combining \eqref{8.1}-\eqref{8.3}, we arrive at
		\begin{equation}\label{8.4}
			\frac{1}{2n} \frac{d}{d\tau} \|w\|_{L^{2n}}^{2n} + \frac{\gamma}{2} \int_0^1 \rho^{\gamma-1} e^s w^{2n} \, dy \le C\left( T \right) (1+\| w\| _{L^{2n}}^{2n}+\left\| u \right\| _{L^{2n}}^{2n}).
		\end{equation}
		Multiplying the second equation of \eqref{4.4} by $r u^{2n-1}$ and integrating over the spatial domain $[0,1]$, we obtain
		\begin{equation}\label{8.5}
			\begin{aligned}
				\frac{1}{2n}&\frac{d}{d\tau}\int_0^1{u^{2n}}\,dy+\int_0^1{\left( \frac{u^{2n}}{r^2}+(2n-1)\rho ^2(\partial _yu)^2u^{2n-2}r^2 \right)}\,dy\\
				&\le \int_0^1{|}\partial _yP|r|u|^{2n-1}\,dy\\
				&\le \int_0^1{\gamma}\rho ^{\gamma -1}(|w|+|u|)|u|^{2n-1}\,dy+\int_0^1{e^s\left| s_y \right|\rho ^{\gamma}r|u|^{2n-1}\,dy}\\
				&\le C\left( T \right) \left( 1+\left\| w \right\| _{L^{2n}}^{2n}+\left\| u \right\| _{L^{2n}}^{2n} \right) .\\
			\end{aligned}
		\end{equation}
		Coupling \eqref{8.4} with \eqref{8.5} yields
		\begin{equation}
			\begin{split}
				\frac{1}{2n}&\frac{d}{d\tau}\left( \| w\| _{L^{2n}}^{2n}+\| u\| _{L^{2n}}^{2n} \right) +\frac{\gamma}{2}\int_0^1{\rho ^{\gamma -1}}e^sw^{2n}\,dy
				\\
				&+\int_0^1{\left( \frac{u^{2n}}{r^2}+(2n-1)\rho ^2(\partial _yu)^2u^{2n-2}r^2 \right)}\,dy\le C\left( T \right) \left( 1+\left\| w \right\| _{L^{2n}}^{2n}+\left\| u \right\| _{L^{2n}}^{2n} \right).
			\end{split}
		\end{equation}
		Applying Gronwall's inequality yields
		$$\underset{0\le t\le T}{\mathrm{sup}}\left( \| w\| _{L^{2n}}^{2n}+\| u\| _{L^{2n}}^{2n} \right) \le \left( \left\| w_0 \right\| _{L^{2n}}^{2n}+\left\| u_0 \right\| _{L^{2n}}^{2n}+2nTC\left( T \right) \right) \exp \left\{ 2nC\left( T \right) T \right\} .$$
		Taking the $(2n)$-th root on both sides of the inequality and using the fact
		$$\underset{n\rightarrow \infty}{\lim}\left( \left\| f \right\| _{L^{2\mathrm{n}}}^{2n}+\left\| g \right\| _{L^{2\mathrm{n}}}^{2n} \right) ^{\frac{1}{2n}}=\max \left( \| f\| _{L^{\infty}},\| g\|_{L^{\infty}} \right),$$
		we arrive at 
		\begin{equation}
			\underset{0\le t\le T}{\mathrm{sup}}\max \left( \| w\| _{L^{\infty}},\| u\| _{L^{\infty}} \right) \le C\left( T \right) \left( \left\| w_0 \right\| _{L^{\infty}}+\left\| u_0 \right\| _{L^{\infty}}+1 \right). 
		\end{equation}
		We conclude the proof.
	\end{proof}
	\subsection{Lower bound estimates for the density}
	Finally, we establish the lower bound of the density to complete the last step of the lower-order estimates.
	\begin{proposition}\label{pro4.5}
		There exists a constant $C(T) > 0$ such that
		\begin{equation} \label{9.1}
			\|\rho^{-1}(t)\|_{L^\infty(\Omega)}\le C(T).
		\end{equation}
		\begin{proof}
			In view of \eqref{88.1} and \eqref{88.2}, we deduce
			\begin{equation}\label{9.2}
				\underset{0\le t\le T}{\mathrm{sup}}\left\| r\partial _y\rho \right\| _{L^{\infty}\left( 0,1 \right)}\le C\left( T \right) .
			\end{equation}
			Let $v = \rho^{-1}$. Since $v$ satisfies the equation $v_\tau = (ru)_y$, we have$$\int_0^1 v(y, \tau) \, dy = \int_0^1 v_0(y) \, dy.$$
			Let $\beta < 1$ be sufficiently small. By virtue of the one-dimensional Sobolev embedding $W^{1,1}(I) \hookrightarrow L^\infty(I)$ and \eqref{9.2}
			$$
			\begin{aligned}
				v^\beta(y, \tau) &\le \int_0^1 v^\beta \, dy + \beta \int_0^1 v^{\beta-1} |\partial_y v| \, dy \\
				&\le \left( \int_0^1 v \, dy \right)^{\beta} + \beta \int_0^1 |\partial_y \rho| v^{\beta+1} \, dy \\
				&\le C(\beta) + \beta \|r\partial_y \rho\|_{L^\infty(0,1)} \left( \int_0^1 v r^{-\frac{3}{2}} \, dy \right)^{\frac{2}{3}} \left( \int_0^1 v^{3\beta+1} \, dy \right)^{\frac{1}{3}} \\
				&\le C(\beta) + C(T)\beta   \|\rho^{-1}(t)\|_{L^\infty(\Omega)} ^\beta.
			\end{aligned}$$
			Choosing $\beta$ sufficiently small s.t.
			\begin{equation}
				\beta =\min \left\{ \frac{1}{2C\left( T \right)},\frac{1}{2} \right\} ,
			\end{equation}
			and taking the supremum over $(y, \tau)$, we arrive at
			$$\|\rho^{-1}(t)\|_{L^\infty(\Omega)}\le C(T).$$
		\end{proof}
	\end{proposition}
	\section{A priori estimates: Lower-Order Estimates for $N=3$}
	The proof for the three-dimensional case is essentially the same as that for the two-dimensional case. However, due to changes in the weighted estimates near the origin, a restriction on $\gamma$ is imposed. Now we present the lower-order estimates for the three-dimensional case.
	\subsection{$L^2$ estimates for velocity and effective velocity}
	\begin{proposition} \label{prop5.1}
		There exists a constant $C(T) > 0$, dependent of $T$, such that
		\begin{equation} \label{50.1}
			\sup_{0 \le \tau \le T} \int_0^1 (u^2 +\rho ^{\gamma -1}) \, dy + \int_0^T \int_0^1 \left( \frac{u^2}{r^2} + \rho^2 (\partial_y u)^2 r^4 \right) \, dy d\tau \le C(T),
		\end{equation}
		\begin{equation} \label{50.2}
			\sup_{0 \le \tau \le T} \int_0^1 (u + 2r^2\partial_y \rho)^2 \, dy + \int_0^T \int_0^1 \rho^{\gamma-1} (r^2\partial_y \rho)^2 \, dy d\tau \le C(T).
		\end{equation}
	\end{proposition}
	
	\begin{proof}
		Multiplying the second equation of \eqref{4.4} with $r^2u$, we get basic energy estimate
		\begin{equation}
			\frac{d}{d\tau} \int_0^1 \left( \frac{1}{2}u^2 + \frac{\rho^{\gamma -1}}{\gamma -1}e^s \right) dy + \int_0^1 \left( 2\rho^2 r^4 (\partial_y u)^2 + \frac{2u^2}{r^2} \right) dy = 0.
		\end{equation}
		Therefore, we prove \eqref{10}.
		\par
		Now multiplying the third equation of \eqref{4.6} with $w=u+2r^2\partial_y \rho$, we get
		$$
		\begin{aligned}
			\frac{1}{2}\frac{d}{d\tau}\int_0^1{w^2}\,dy&=-\int_0^1{\partial _yP}(r^2u)\,dy-2\int_0^1{\partial _yP}(r^4\partial _y\rho )\,dy
			\\
			&=\int_0^1{P}\partial _y(r^2u)\,dy-2\int_0^1{\partial _yP}\left( r^4\partial _y\rho \right) \,dy
			\\
			&=-\frac{d}{d\tau} \int_0^1 \frac{\rho^{\gamma-1}}{\gamma-1} e^{s} \, dy-2\int_0^1{\partial _yP}\left( r^4\partial _y\rho \right) \,dy,
		\end{aligned}
		$$
		which is equivalent to  
		\begin{equation}\label{554.6}
			\frac{d}{d\tau} \left( \int_0^1 \frac{1}{2}w^2 \, dy + \int_0^1 \frac{\rho^{\gamma-1}}{\gamma-1} e^{s} \, dy \right) + \int_0^1 2\gamma r^4 \rho^{\gamma-1} e^s (\partial_y \rho)^2 \, dy = - \int_0^1 2 r^4 \rho^\gamma e^s (\partial_y s) (\partial_y \rho) \, dy.
		\end{equation}
		Using Young's inequality and absorbing the first term of source term into the dissipation on the left-hand side, we obtain
		\begin{equation}
			\frac{d}{d\tau} \left( \int_0^1 \frac{1}{2}w^2 \, dy + \int_0^1 \frac{\rho^{\gamma-1}}{\gamma-1} e^s \, dy \right)+ \int_0^1 \gamma r^4 \rho^{\gamma-1} e^s (\partial_y \rho)^2 \, dy \le C \|\rho r^2\|_{L^\infty}^2,
		\end{equation}
		Since
		\begin{equation}
			\left( \rho r^2 \right) \left( y \right) =\left( \rho r^2 \right) \left( y \right) -\left( \rho r^2 \right) \left( 0 \right) \le \int_0^1{\left| \left( \rho r^2 \right) _y \right|dy}\le \int_0^1{\left| \rho _yr^2 \right|+\left| 2rr_y\rho \right|dy},
		\end{equation}
		so we arrive at 
		$$\begin{aligned}
			\|\rho r^2\| _{L^{\infty}}^{2}&\le C\left( \| \rho _yr^2\| _{L^1}+\left\| \rho rr_y \right\| _{L^1} \right) ^2\\
			&=C\left( \| \rho _yr^2\| _{L^1}^{2}+\| \frac{1}{r} \| _{L^1}^{2} \right),
		\end{aligned}$$
		where we used the relation $\partial_y r = \frac{1}{\rho r^2}$. Now we estimate each term on the RHS. First, by Hölder's inequality, we get
		$$\|\rho_y r^2\|_{L^1}^2 = \left( \int_0^1 |\rho_y r^2| \, dy \right)^2 \le \left( \int_0^1 |\rho_y r^2|^2 \, dy \right) \left( \int_0^1 1 \, dy \right) = \|\rho_y r^2\|_{L^2}^2.$$
		Second, for the term $\left\| \frac{1}{r} \right\|_{L^1}$, using the coordinate transformation $dy = \rho r dr$, we have
		\begin{equation}
			\begin{aligned}
				\| \frac{1}{r} \| _{L^1}^{2}&=\left( \int{\frac{1}{r}}\,dy \right) ^2=\left( \int{\rho}r\,dr \right) ^2\\
				&=\left( \int{\rho ^{\frac{1}{2}}}r^{\xi +\frac{1}{2}}\cdot \rho ^{\frac{1}{2}}r^{\frac{1}{2}-\xi}\,dr \right) ^2\\
				&\le \| \rho ^{\frac{1}{2}}r^{\xi +\frac{1}{2}} \| _{L^{\infty}(0,R)}^{2}\left( \int{\rho ^{\frac{1}{2}}}r^{\frac{1}{\gamma}}r^{\frac{1}{2}-\frac{1}{\gamma}-\xi}\,dr \right) ^2\\
				&\le \| \rho ^{\frac{1}{2}}r^{{\xi}+\frac{1}{2}} \| _{L^{\infty}(0,R)}^{2}(\int{\rho ^{\gamma}}r^2dr)^{\frac{1}{\gamma}}(\int{r^{\frac{\gamma -2}{2\gamma -1}-\frac{2\gamma}{2\gamma -1}\xi}}dr)^{\frac{2\gamma -1}{\gamma}}\\
				&\le C(\xi )\| \rho ^{\frac{1}{2}}r^{{\xi}+\frac{1}{2}} \| _{L^{\infty}(0,R)}^{2},\\
			\end{aligned}  
		\end{equation}
		where $\xi$ is small enough s.t. 
		$$\xi <\frac{3}{2}-\frac{3}{2\gamma}.$$
		By Lemma \ref{lemma 3D}, in three-dimensional space, the weighted estimate at the center of the sphere is
		$$
		\| \rho ^{\frac{1}{2}}r^{\frac{1}{2}+\xi}\| _{L^{\infty}(0,R)}\le C(\xi )(\| \nabla \rho ^{\frac{1}{2}}\| _{L^2}+1)\\=C(\xi )\left( 1+\| \rho _yr^2\|_{L^2(0,1)} \right) .\\
		$$

		By the same argument as in the two-dimensional case, we derive
		\begin{equation}
			\begin{split}
				\frac{d}{d\tau}&\int_0^1{\left( \frac{1}{2}u^2+\frac{1}{2}w^2+\frac{\rho^{\gamma -1}}{\gamma -1}e^s \right)}dy+\int_0^1{\left( 2\rho ^2r^4(\partial _yu)^2+\frac{2u^2}{r^2} \right)+\gamma r^4 \rho^{\gamma-1} e^s (\partial_y \rho)^2}dy
				\\
				&\le C(\xi )\left( 1+\| \rho _yr^2\| _{L^2(0,1)}^{2} \right) .
				\\
				&\le C(\xi )\left( 1+\left\| u \right\| _{L^2(0,1)}^{2}+\left\| w \right\| _{L^2(0,1)}^{2} \right).
			\end{split}
		\end{equation}
		Using Gronwall's inequality, we prove \eqref{50.1} and \eqref{50.2}. Consequently, we also obtain
		\begin{equation}\label{44.99}
			\|\rho^{\frac{1}{2}}r^{{\xi}+\frac{1}{2}}\|_{L^{\infty}(0,R)} \le C(T),
		\end{equation}
		and 
		\begin{equation}
			\|\rho_y r^2\|_{L^2(0,1)} \leq C(T).
		\end{equation}
	\end{proof}
	\subsection{Upper Bound Estimates for the Density}
	For the case $1 < \gamma < 2$, the problem can be solved using the same arguments as in the two-dimensional case. However, since the weighted estimates near the origin are insufficient in the three-dimensional case, we employ an alternative approach to establish the result for $1 < \gamma < 3$.
	\begin{proposition}\label{8.1}
		For $1<\gamma<3$, there exists a constant $C > 0$, dependent of $T$, such that
		\begin{equation}
			\sup_{0 \le t \le T} \| \rho \|_{L^\infty(0,R)} \le C(T).
		\end{equation}
	\end{proposition}
	\begin{proof}
		Multiplying the third equation of \eqref{4.4} by $r^2u^{2n-1}$($n$ will be determined later) and integrating by parts, we arrive at
		\begin{equation}\label{55.13}
			\begin{aligned}
				\frac{1}{2n}&\frac{d}{d\tau}\int_0^1{u^{2n}}\,dy+\int_0^1{\left( 8\frac{u^{2n}}{r^2}+2(2n-1)\rho ^2r^4(\partial _yu)^2u^{2n-2} \right)}\,dy
				\\
				&\le C\int_0^1{\left( (2n-1)\rho ^{\gamma}e^s r^2u^{2n-2}\partial _yu+2\rho ^{\gamma -1}e^s\frac{u^{2n-1}}{r} \right)}\,dy\triangleq F_1+F_2.
			\end{aligned}
		\end{equation}
		Using Young’s inequality, we obtain 
		\begin{equation}
			F_1\le \int_0^1{\rho ^2r^4(\partial _yu)^2u^{2n-2}+\frac{u^{2n}}{r^2}}\,dy+C\int_0^1{\rho ^{2n\left( \gamma -1 \right)}r^{2\left( n-1 \right)}}\,dy.
		\end{equation}
		$F_2$ can be estimated by 
		\begin{equation}\label{55.15}
			F_2\le \int_0^1{\frac{u^{2n}}{r^2}}\,dy+C\int_0^1{\rho ^{2n\left( \gamma -1 \right)}r^{2\left( n-1 \right)}}\,dy.
		\end{equation}
		Combining \eqref{55.13}-\eqref{55.15}, we get
		\begin{equation}\label{55.16}
			\frac{1}{2n}\frac{d}{d\tau}\int_0^1{u^{2n}}\,dy+\int_0^1{\left( 8\frac{u^{2n}}{r^2}+2(2n-1)\rho ^2r^4(\partial _yu)^2u^{2n-2} \right)}\,dy\le C\int_0^1{\rho ^{2n\left( \gamma -1 \right)}r^{2\left( n-1 \right)}}\,dy.
		\end{equation}
		Let $R_T= \sup_{0 \le t \le T} \|\rho(t)\|_{L^\infty(\Omega)} + 1.$ Fix $n \ge 4$ s.t. 
		\begin{equation}
			\gamma <3-\frac{1}{n},
		\end{equation}
		and $\sigma$ be a constant $0\le \sigma <1-\frac{1}{n}$ satisfy
		\begin{equation}\label{44.117}
			1-\frac{1}{2n}+\sigma <\gamma <2+\sigma.
		\end{equation}
		
		In Euler coordinates, \eqref{55.16} is equivalent to 
		\begin{equation}
			\begin{aligned}
				\mathop {\mathrm{sup}}_{0\le t\le T}&\int_0^R{\rho}u^{2n}r^2\,dr+\int_0^T{\int_0^R{\left( \rho u^{2n}+\rho (\partial _ru)^2u^{2n-2}r^2 \right)}}\,dr\,dt
				\\
				&\le C+C\int_0^T{\int_0^R{\rho ^{2n(\gamma -1)+1}}}r^{2n}\,dr\,dt
				\\
				&\le C+C\mathop {\mathrm{sup}} \limits_{0\le t\le T}\left\| \rho ^{\frac{1}{2}}r^{\frac{1}{2}+\xi} \right\| _{L^{\infty}(0,R)}^{2\left( 2n(\gamma -1)+1-2n\sigma \right)}R_{T}^{2n\sigma}\int_0^T{\int_0^R{r^{2n-2\left( 2n(\gamma -1)+1-2n\sigma \right) (\frac{1}{2}+\xi )}}}\,dr\,dt
				\\
				&\le C(T)R_{T}^{2n\sigma}.
			\end{aligned}
		\end{equation}
		Let $\xi$ small enough s.t. 
		\begin{equation}
			\xi <\frac{2n-2n\left( \gamma -1 \right) +2n\sigma}{2\left[ 2n\left( \gamma -1 \right) +1-2n\sigma \right]}.
		\end{equation}
		The existence of such a $\xi$ is guaranteed by \eqref{44.117}.
		Multiplying the third equation of \eqref{4.6} by $w^{2n-1}$ and integrating over $[0,1]$ gives 
		\begin{equation}\label{66.111}
			\begin{aligned}
				\frac{1}{2n}\frac{d}{d\tau}\int_0^1{w^{2n}}\,dy&+\int_0^1{\gamma \rho ^{\gamma -1}\frac{w^{2n}}{2}e^s}\,dy\\
				&=\int_0^1{\left( \gamma \rho ^{\gamma -1}\frac{u}{2}w^{2n-1}e^s-r^2w^{2n-1}\rho ^{\gamma}\partial _ys \right)}\,dy\\
				&\le \frac{\gamma}{2}\int_0^1{w^{2n}\rho ^{\gamma -1}e^s}\,dy+C\int_0^1{u^{2n}\rho ^{\gamma -1}e^s}\,dy+C\int_0^1{r^{4n}\rho ^{2n+\gamma -1}}\,dy\\
				&\le \frac{\gamma}{2}\int_0^1{w^{2n}\rho ^{\gamma -1}e^s}\,dy+C\int_0^1{u^{2n}\rho ^{\gamma -1}e^s}\,dy\\
				&\quad +C\| \rho ^{\frac{1}{2}}r^{\xi +\frac{1}{2}} \| _{L^{\infty}}^{4n+2\gamma -2}\int_0^1{r^{2n-\gamma +1-\left( 4n+2\gamma -2 \right) \xi}}\,dy.\\
			\end{aligned}
		\end{equation}
		Since $n>\gamma/2$, we have 
		\begin{equation}
			2n-\gamma +1>1,
		\end{equation}
		then let $\xi$ small enough s.t.
		\begin{equation}
			\xi \le \frac{1}{4n+2\gamma -2}.
		\end{equation}
		Therefore, we arrive at 
		\begin{equation}\label{66.114}
			\frac{1}{2n}\frac{d}{d\tau}\int_0^1{w^{2n}}\,dy+\int_0^1{\frac{\gamma}{2}\rho ^{\gamma -1}\frac{w^{2n}}{2}e^s}\,dy\le C\int_0^1{u^{2n}\rho ^{\gamma -1}e^s}\,dy+C\left( T,R \right) .
		\end{equation}
		Integrating \eqref{66.114} and transforming into Euler coordinates, then using $\underline{s} \le s \le \bar{s}$, we obtain
		\begin{equation}
			\begin{aligned}
				&\mathop {\mathrm{sup}} \limits_{0\le t\le T}\int_0^R{\left| \partial _r\rho ^{\frac{1}{2n}} \right|^{2n}}r^2\,dr+\int_0^T{\int_0^R{\rho ^{\gamma}\left( u+2\rho ^{-1}\rho _r \right) ^{2n}r^2}}\,dr\,dt\\
				&\quad \le C\mathop {\mathrm{sup}} \limits_{0\le t\le T}\int_0^R{\rho}u^{2n}r^2\,dr+C\mathop {\mathrm{sup}} \limits_{0\le t\le T}\int_0^R{\rho}\left( u+2\rho ^{-1}\partial _r\rho \right) ^{2n}r^2\,dr\\
				&\quad \quad +\int_0^T{\int_0^R{\rho ^{\gamma}\left( u+2\rho ^{-1}\rho _r \right) ^{2n}r^2}}\,dr\,dt\\
				&\quad \le C(T)R_{T}^{2n\sigma}+C\int_0^T{\int_0^R{\rho ^{\gamma}}}u^{2n}r^2\,dr\,dt\\
				&\quad \le C(T)R_{T}^{2n\sigma}+C\int_0^T{\int_0^R{\rho}}u^{2n}\rho ^{\gamma -1}r^2\,dr\,dt\\
				&\quad \le C(T)R_{T}^{2n\sigma}+C\mathop {\mathrm{sup}} \limits_{0\le t\le T}\| \rho ^{\frac{1}{2}}r^{\frac{1}{2}+\xi} \| _{L^{\infty}(0,R)}^{2\left( \gamma -1 \right)}\int_0^T{\int_0^R{\rho}}u^{2n}\,dr\,dt\\
				&\quad \le C(T)R_{T}^{2n\sigma},\\
			\end{aligned}
		\end{equation}
		taking $\xi$ small enough s.t.
		\begin{equation}
			\xi \le \frac{3-\gamma}{2\left( \gamma -1 \right)}.
		\end{equation}
		The existence of such a $\xi$ is guaranteed by the condition $\gamma < 3$.
		
		Now choosing $\beta$ s.t. 
		\begin{equation}\label{55.23}
			\max\{\frac{1}{2n},\sigma\}< \beta <\left(1 - \frac{1}{n}\right).
		\end{equation}
		Using one-dimensional Sobolev embedding and \eqref{5.58}
		$$\begin{aligned}
			\| \rho ^{\beta}\| _{L^{\infty}(0,R)}&\le C\int_0^R{\rho ^{\beta}}\,dr+C\int_0^R{|}\partial _r\rho ^{\beta}|\,dr\\
			&\le C\int_0^R{(}\rho ^3r^2)^{\frac{\beta}{3}}r^{-\frac{2\beta}{3}}\,dr+C\int_0^R{\left| \partial _r\rho ^{\frac{1}{2n}} \right|}\rho ^{\beta -\frac{1}{2n}}\,dr\\
			&\le C\left( \int_0^R{\rho ^3}r^2\,dr \right) ^{\frac{\beta}{3}}\left( \int_0^R{r^{-\frac{2\beta}{3-\beta}}}\,dr \right) ^{\frac{3-\beta}{3}}\\
			&\quad +C\left( \int_0^R{\left| \partial _r\rho ^{\frac{1}{2n}} \right|^{2n}}r^2\,dr \right) ^{\frac{1}{2n}}\left( \int_0^R{\rho ^{\frac{2n}{2n-1}(\beta -1)+1}}r^{-\frac{2}{2n-1}}\,dr \right) ^{\frac{2n-1}{2n}}\\
			&\le C+C(T)R_{T}^{\sigma}\mathop {\mathrm{sup}} \limits_{0\le t\le T}\left\| \rho ^{\frac{1}{2}}r^{\frac{1}{2}+\xi} \right\| _{L^{\infty}}^{2\left( (\beta -1)+\frac{2n-1}{2n} \right)}\\
			&\quad \cdot \left( \int_0^R{r^{-\frac{2}{2n-1}-\left( \frac{1}{2}+\xi \right) 2\left( \frac{2n}{2n-1}(\beta -1)+1 \right)}}\,dr \right) ^{\frac{2n-1}{2n}}\\
			&\le C(T)R_{T}^{\sigma},\\
		\end{aligned}$$
		where $\xi$ is taken sufficiently small such that
		$$-\frac{2}{2n-1}-\left( \frac{1}{2}+\xi \right) 2\left( \frac{2n}{2n-1}(\beta -1)+1 \right) \ge -1.$$
		Such a $\xi$ exists by virtue of \eqref{55.23}, so we obtain
		$$R_T^\beta \leq C(T).$$
		This establishes the upper bound of the density.
	\end{proof}
	\subsection{$L^{\infty}$ estimates for velocity and effective velocity}
	\begin{proposition}
		There exists a constant $C(T) > 0$ such that
		\begin{equation}\label{888.1}
			\sup_{0\le \tau \le T} \|w(\cdot, \tau)\|_{L^\infty(0,1)} \le C(T),
		\end{equation}
		\begin{equation}\label{888.2}
			\sup_{0\le \tau \le T} \|u(\cdot, \tau)\|_{L^\infty(0,1)} \le C(T),
		\end{equation}
		and
		\begin{equation}\label{888.3}
			\underset{0\le t\le T}{\mathrm{sup}}\left\| r^2\partial _y\rho \right\| _{L^{\infty}\left( 0,1 \right)}\le C\left( T \right) .
		\end{equation}
	\end{proposition}
	The proof of this proposition is similar to that of Proposition \ref{pro4.4}, so we omit it.
	\subsection{Lower Bound Estimate for the Density}
	\begin{proposition}\label{Pro5.4}
		There exists a constant $C(T) > 0$ such that
		\begin{equation} \label{44.266}
			\|\rho^{-1}(t)\|_{L^\infty(\Omega)}\le C(T).
		\end{equation}
	\end{proposition}
	\begin{proof}
		Let \(\beta \in (0,1)\) be a parameter to be determined later. For any \((y,\tau) \in (0,1)\times(0, T)\), using the one‑dimensional Sobolev embedding, Hölder's inequality, \eqref{888.3}, we have
		$$
		\begin{aligned}
			v^\beta(y,\tau)&\leq \int_0^1 v^\beta dy+\beta\int_0^1 v^{\beta-1}|\partial_y v|dy\\
			&\leq \left(\int_0^1 v dy\right)^{\beta}+\beta\int_0^1 v^{\beta+1}|\partial_y\rho|dy\\
			&\leq C(\beta)+\beta\|{r^{2}\partial_y\rho}\|_{L^\infty(0,1)}\left(\int_0^1 vr^{-\frac{8}{3}}dy\right)^{\frac{3}{4}}\left(\int_0^1 v^{4\beta+1}dy\right)^{\frac{1}{4}}\\
			&\le C(\beta) + C(T)\beta   \|\rho^{-1}(t)\|_{L^\infty(\Omega)} ^\beta.
		\end{aligned}
		$$
		Taking the supremum over \((y,\tau)\) of the above inequality and setting \(\beta = \min\{\frac{1}{2C(T)}, \frac{1}{2}\}\), we obtain \eqref{44.266}. This completes the proof of Proposition \ref{Pro5.4}.
	\end{proof}
	\section{A priori estimates: Higher-Order estimates}
	\begin{proposition}
		Let $N = 2$ or $N = 3$  and assume that
		\begin{equation}\label{66.1}
			\begin{aligned}
				\mathop {\mathrm{sup}} \limits_{0\le t\le T}&\left( \| \mathbf{u}\| _{L^4} +\| \rho \| _{L^{\infty}}+\| \rho ^{-1}\| _{L^{\infty}}+\| \nabla \rho \| _{L^2}+\| \nabla \rho \| _{L^4} \right) 
				\\
				&+\int_0^T{\left( \| \nabla \rho \| _{L^2}^{2}+\| \nabla \rho \| _{L^4}^{4}+\|\nabla \mathbf{u}\| _{L^2}^{2} \right)}dt\le C(T).
			\end{aligned}
		\end{equation}
		Then there exists a constant $C(T)>0$, such that
		\begin{equation}\label{666.2}
			\sup_{0\le t\le T} \|\nabla \mathbf{u}\|_{L^2}^2 + \int_0^T \int_\Omega \rho |\dot{\mathbf{u}}|^2 \, dx \, dt \le C(T). 
		\end{equation}
	\end{proposition} 
	\begin{proof}
		In the radially symmetric setting, we have $\mathcal{D}\mathbf{u} = \nabla \mathbf{u}$. Therefore, we rewrite \eqref{1.1}$_2$ as
		$$\rho\dot{\mathbf{u}} + \nabla P = 2\nabla \cdot (\rho \nabla \mathbf{u})$$
		Multiplying the equation by $\dot{\mathbf{u}}$ and integrating by parts over $\Omega$ yields
		\begin{equation}\label{666.333}
			\int_{\Omega}{\rho}|\dot{\mathbf{u}}|^2\,dx=-\int_{\Omega}{\nabla}P\cdot \dot{\mathbf{u}}\,dx-\int_{\Omega}{2\rho \nabla \mathbf{u}:\nabla \dot{\mathbf{u}}}\,dx\triangleq I_1+I_2.
		\end{equation}
		Since we have 
		$$\|s_y\|_{L^\infty} = \| \frac{1}{\rho r^{N-1}} |\nabla s| \|_{L^\infty},$$
		this implies $\|\nabla s\|_{L^\infty} \le C(T)$.
		Hölder's inequality implies
		\begin{equation}\label{66.23}
			\begin{aligned}
				I_1&\le C(T)\int_{\Omega}{|}\nabla \rho |\sqrt{\rho}|\dot{\mathbf{u}}|\,dx+C(T)\int_{\Omega}{|}\nabla s|\sqrt{\rho}|\dot{\mathbf{u}}|\,dx\\
				&\le \frac{1}{8}\| \sqrt{\rho}\dot{\mathbf{u}}\| _{L^2}^{2}+C(T)\| \nabla \rho \| _{L^2}^{2}+C\left( T \right) .\\
			\end{aligned} 
		\end{equation}
		It follows from \eqref{1.1}$_2$ that $u$ solves the elliptic system
		\begin{equation}\label{66.55}
			\begin{cases}
				\Delta \mathbf{u}= \frac{1}{2\rho} (\rho\dot{\mathbf{u}} + \nabla P - 2\nabla\rho \cdot \nabla \mathbf{u}), \\
				\mathbf{u}|_{\partial\Omega} = 0.
			\end{cases}
		\end{equation}
		Standard $L^2$ estimates for elliptic systems yield
		$$
		\begin{aligned}
			\| \mathbf{u}\|_{H^2} &\le C(T)\left( \| \dot{\mathbf{u}}\|_{L^2} + \| \nabla \rho \|_{L^2} + \| \nabla s\|_{L^2} + \| |\nabla \rho ||\nabla \mathbf{u}|\|_{L^2} \right) \\
			&\le C(T)\left( \| \sqrt{\rho}\dot{\mathbf{u}}\|_{L^2} + \| \nabla \rho \|_{L^2} +1+ \| \nabla \rho \|_{L^4}\| \nabla \mathbf{u}\|_{L^4} \right) \\
			&\le C(T)\left( \| \sqrt{\rho}\dot{\mathbf{u}}\|_{L^2} + \| \nabla \rho \|_{L^2} +1+ \| \nabla \mathbf{u}\|_{L^2}^{\frac{4-N}{4}}\| \nabla \mathbf{u}\|_{H^1}^{\frac{N}{4}} \right) \\
			&\le \frac{1}{2}\| \mathbf{u}\|_{H^2} + C(T)\left( \| \sqrt{\rho}\dot{\mathbf{u}}\|_{L^2} + \| \nabla \rho \|_{L^2} +1+ \| \nabla \mathbf{u}\|_{L^2} \right).
		\end{aligned}$$
		Therefore, we arrive at
		\begin{equation}\label{66.2}
			\|\mathbf{u}\|_{H^2} \le C(T)\left( \|\sqrt{\rho}\dot{\mathbf{u}}\|_{L^2} + \|\nabla\rho\|_{L^2} + \|\nabla \mathbf{u}\|_{L^2}+1 \right).
		\end{equation}
		Using \eqref{66.1} and \eqref{66.2}, we get
		\begin{equation}\label{66.5}
			\begin{aligned}
				I_2 &= -\int_{\Omega} 2\rho \nabla \mathbf{u} : \nabla \mathbf{u}_t \, dx - \int_{\Omega} 2\rho \nabla \mathbf{u} : \nabla (\mathbf{u} \cdot \nabla \mathbf{u}) \, dx \\
				&= -\frac{d}{dt} \int_{\Omega} \rho |\nabla \mathbf{u}|^2 \, dx + \int_{\Omega} \partial_t \rho |\nabla \mathbf{u}|^2 \, dx - \int_{\Omega} 2\rho \nabla \mathbf{u} : \nabla (\mathbf{u} \cdot \nabla \mathbf{u}) \, dx \\
				&= -\frac{d}{dt} \int_{\Omega} \rho |\nabla \mathbf{u}|^2 \, dx - \int_{\Omega} \mathrm{div}(\rho \mathbf{u}) |\nabla \mathbf{u}|^2 \, dx - \int_{\Omega} 2\rho \nabla \mathbf{u} : \nabla (\mathbf{u} \cdot \nabla \mathbf{u}) \, dx \\
				&\le -\frac{d}{dt} \int_{\Omega} \rho |\nabla \mathbf{u}|^2 \, dx + C(T) \int_{\Omega} |\nabla \rho| |\mathbf{u}| |\nabla \mathbf{u}|^2 \, dx + C(T) \int_{\Omega} |\mathbf{u}| |\nabla \mathbf{u}| |\nabla^2 \mathbf{u}| \, dx \\
				&\le -\frac{d}{dt} \int_{\Omega} \rho |\nabla \mathbf{u}|^2 \, dx + C(T) \| \nabla \rho \|_{L^4} \| \mathbf{u} \|_{L^4} \| \nabla \mathbf{u} \|_{L^4}^2 \\
				&\quad + C(T) \| \mathbf{u} \|_{L^4} \| \nabla^2 \mathbf{u} \|_{L^2} \| \nabla \mathbf{u} \|_{L^4} \\
				&\le -\frac{d}{dt} \int_{\Omega} \rho |\nabla \mathbf{u}|^2 \, dx + C(T) \| \nabla \mathbf{u} \|_{L^4}^2 + C(T) \| \nabla^2 \mathbf{u} \|_{L^2} \| \nabla \mathbf{u} \|_{L^4} \\
				&\le -\frac{d}{dt} \int_{\Omega} \rho |\nabla \mathbf{u}|^2 \, dx + C(T) \| \nabla \mathbf{u} \|_{L^2}^{\frac{4-N}{2}} \| \nabla \mathbf{u} \|_{H^1}^{\frac{N}{2}} + C(T) \| \nabla \mathbf{u} \|_{L^2}^{\frac{4-N}{4}} \| \nabla \mathbf{u} \|_{H^1}^{\frac{N+4}{4}} \\
				&\le -\frac{d}{dt} \int_{\Omega} \rho |\nabla \mathbf{u}|^2 \, dx + \frac{1}{8} \| \sqrt{\rho} \dot{\mathbf{u}} \|_{L^2}^2 + C(T) \| \nabla \rho \|_{L^2}^2 + C(T) \| \nabla \mathbf{u} \|_{L^2}^2+C(T).
			\end{aligned}
		\end{equation}  
		Substituting \eqref{66.23} and \eqref{66.5} into
		\eqref{666.333} and integrating with respect to time $t$, we arrive at \eqref{666.2}.
	\end{proof}
	\begin{proposition}
		Under the assumption \eqref{66.1}, there exists a constant $C(T)$, such that
		\begin{equation}\label{666.7}
			\sup_{0 \le t \le T} \left( \|\sqrt{\rho}\dot{\mathbf{u}}\|_{L^2}^2 + \|\mathbf{u}\|_{H^2}^2 \right) + \int_0^T \int_\Omega |\nabla\dot{\mathbf{u}}|^2 \, dx \, dt \le C(T)
		\end{equation}
	\end{proposition}
	\begin{proof}
		Applying $\dot{\mathbf{u}}^j \left[ \partial/\partial t + \nabla \cdot (\mathbf{u} \cdot) \right]$ to \eqref{1.1}$_2$ and integrating by parts yields
		\begin{equation}\label{66.68}
			\begin{aligned}
				\left( \frac{1}{2}\int_{\Omega}\rho|\dot{\mathbf{u}}|^2\,dx \right)_t &= -\int_{\Omega}\dot{\mathbf{u}}^j[\partial_j P_t + \nabla \cdot (\mathbf{u}\partial_j P)]\,dx \\
				&\quad + \int_{\Omega}\dot{\mathbf{u}}^j[\partial_t\partial_i(2\rho \partial_i \mathbf{u}^j) + \nabla \cdot (\mathbf{u}\partial_i(2\rho \partial_i \mathbf{u}^j))]\,dx \\
				&\triangleq \sum_{i=1}^2 N_i.
			\end{aligned}
		\end{equation}
		
		Using integration by parts and the equation
		$$
		P_t + \mathbf{u} \cdot \nabla P + \gamma P\nabla \cdot\mathbf{u} = 0,
		$$
		$N_1$ can be estimated by
		\begin{equation}\label{66.69}
			\begin{aligned}
				N_1 &= -\int_{\Omega}\dot{\mathbf{u}}^j[\partial_j P_t + \operatorname{div}(\partial_j P\mathbf{u})]\,dx \\
				&= \int_{\Omega} \left[ -\partial_j \dot{\mathbf{u}}^j \mathbf{u}\cdot \nabla P - \gamma \partial_j \dot{\mathbf{u}}^j P\nabla \cdot \mathbf{u} - P\partial_j (\partial_k \dot{\mathbf{u}}^j \mathbf{u}^k) \right] \,dx \\
				&= \int_{\Omega} \left[ \partial_k \left( \partial_j \dot{\mathbf{u}}^j \mathbf{u}^k \right) P - \gamma \partial_j \dot{\mathbf{u}}^j P\nabla \cdot \mathbf{u} - P\partial_j (\partial_k \dot{\mathbf{u}}^j \mathbf{u}^k) \right] \,dx \\
				&\le C(T) \| \nabla \mathbf{u} \|_{L^2} \| \nabla \dot{\mathbf{u}} \|_{L^2} \\
				&\le \varepsilon \| \nabla \dot{\mathbf{u}} \|_{L^2}^{2} + C(\varepsilon, T) \| \nabla \mathbf{u} \|_{L^2}^{2}.
			\end{aligned}
		\end{equation}
		Integrating by parts yields
		\begin{equation}
			\begin{aligned}
				N_2&=\int_{\Omega}{\mathbf{\dot{u}}^j}\left[ \partial _t\partial _i(2\rho \partial _i\mathbf{u}^j)+\nabla \cdot (\mathbf{u}\partial _i(2\rho \partial _i\mathbf{u}^j)) \right] \,dx\\
				&=-\int_{\Omega}{\partial _i}\mathbf{\dot{u}}^j\partial _t(2\rho \partial _i\mathbf{u}^j)\,dx-\int_{\Omega}{\partial _k}\mathbf{\dot{u}}^j\mathbf{u}^k\partial _i(2\rho \partial _i\mathbf{u}^j)\,dx\\
				&=-2\int_{\Omega}{\rho}|\nabla \mathbf{\dot{u}}|^2\,dx+2\int_{\Omega}{\rho}\partial _i\mathbf{\dot{u}}^j\partial _i(\mathbf{u}^k\partial _k\mathbf{u}^j)\,dx\\
				&\quad -2\int_{\Omega}{\partial _i}\mathbf{\dot{u}}^j\partial _t\rho \partial _i\mathbf{u}^j\,dx-2\int_{\Omega}{\partial _k}\mathbf{\dot{u}}^j\mathbf{u}^k\partial _i\rho \partial _i\mathbf{u}^j\,dx\\
				&\quad -2\int_{\Omega}{\partial _k}\mathbf{\dot{u}}^j\mathbf{u}^k\rho \partial _{ii}\mathbf{u}^j\,dx\\
				&\le -2\int_{\Omega}{\rho}|\nabla \mathbf{\dot{u}}|^2\,dx+C(T)\int_{\Omega}{\left( |\nabla \mathbf{\dot{u}}||\nabla \mathbf{u}|^2+|\nabla \mathbf{\dot{u}}||\mathbf{u}||\nabla ^2\mathbf{u}|+|\nabla \mathbf{\dot{u}}||\mathbf{u}||\nabla \rho ||\nabla \mathbf{u}| \right)}\,dx\\
				&\le -2\int_{\Omega}{\rho}|\nabla \mathbf{\dot{u}}|^2\,dx+\varepsilon \int_{\Omega}{|}\nabla \mathbf{\dot{u}}|^2\,dx\\
				&\quad +C(\varepsilon ,T)\int_{\Omega}{\left( |\mathbf{u}|^2|\nabla ^2\mathbf{u}|^2+|\nabla \mathbf{u}|^4+|\nabla \rho |^2|\mathbf{u}|^2|\nabla \mathbf{u}|^2 \right)}\,dx.\\
			\end{aligned}
		\end{equation}
		Standard $L^4$ estimates for elliptic systems yield
		\begin{equation}\label{666.11}
			\begin{aligned}
				\| \mathbf{u} \|_{W^{2,4}} &\le C(T)\left( \| \dot{\mathbf{u}} \|_{L^4} + \| \nabla \rho \|_{L^4} + \| \nabla s \|_{L^4} + \| |\nabla \rho ||\nabla \mathbf{u}| \|_{L^4} \right) \\
				&\le C(T)\left( \| \sqrt{\rho}\dot{\mathbf{u}} \|_{L^4} + \| \nabla \rho \|_{L^4} + 1 + \| \nabla \rho \|_{L^4} \| \nabla \mathbf{u} \|_{L^{\infty}} \right) \\
				&\le C(T)\left( \| \sqrt{\rho}\dot{\mathbf{u}} \|_{L^4} + \| \nabla \rho \|_{L^4} + 1 + \| \nabla \mathbf{u} \|_{L^2}^{\frac{4-N}{N+4}} \| \nabla \mathbf{u} \|_{W^{1,4}}^{\frac{2N}{N+4}} \right) \\
				&\le \frac{1}{2} \| \mathbf{u} \|_{W^{2,4}} + C(T)\left( \| \sqrt{\rho}\dot{\mathbf{u}} \|_{L^4} + \| \nabla \rho \|_{L^4} + 1 + \| \nabla \mathbf{u} \|_{L^2} \right),
			\end{aligned}
		\end{equation}
		so we get $\| \mathbf{u} \|_{W^{2,4}} \le C(T)\left( \| \sqrt{\rho}\dot{\mathbf{u}} \|_{L^4} + \| \nabla \rho \|_{L^4} + 1 + \| \nabla \mathbf{u} \|_{L^2} \right).$ Based on the $L^4$ estimates, we obtain 
		\begin{equation}
			\begin{aligned}
				C(\varepsilon, T)\int_{\Omega} |\mathbf{u}|^2 |\nabla^2\mathbf{u}|^2 \,dx 
				&\le C(\varepsilon, T) \| \mathbf{u} \|_{L^4}^2 \| \nabla^2\mathbf{u} \|_{L^4}^2 \\
				&\le C(\varepsilon, T) \| \mathbf{u} \|_{L^4}^2 \| \mathbf{u} \|_{W^{2,4}}^2 \\
				&\le C(\varepsilon, T) \| \mathbf{u} \|_{L^4}^2 \left( \| \sqrt{\rho}\dot{\mathbf{u}} \|_{L^4}^2 + \| \nabla \rho \|_{L^4}^2 + \| \nabla \mathbf{u} \|_{L^2}^2 + 1 \right) \\
				&\le C(\varepsilon, T) \left( \| \sqrt{\rho}\dot{\mathbf{u}} \|_{L^2}^{\frac{4-N}{2}} \| \nabla \dot{\mathbf{u}} \|_{L^2}^{\frac{N}{2}} + \| \nabla \rho \|_{L^4}^2 + \| \nabla \mathbf{u} \|_{L^2}^2 + 1 \right) \\
				&\le \varepsilon \| \nabla \dot{\mathbf{u}} \|_{L^2}^2 + C(\varepsilon, T) \left( \| \sqrt{\rho}\dot{\mathbf{u}} \|_{L^2}^2 + \| \nabla \rho \|_{L^4}^2 + \| \nabla \mathbf{u} \|_{L^2}^2 + 1 \right).
			\end{aligned}
		\end{equation}
		$C(\varepsilon ,T)\int_{\Omega}{\left( |\nabla \mathbf{u}|^4 \right)}\,dx$ can be estimated by 
		\begin{equation}
			\begin{aligned}
				C(\varepsilon, T)\int_{\Omega} |\nabla \mathbf{u}|^4 \,dx 
				&\le C(\varepsilon, T) \| \nabla \mathbf{u} \|_{H^1}^4 \\
				&\le C(\varepsilon, T) \left( \| \sqrt{\rho}\dot{\mathbf{u}} \|_{L^2}^4 + \| \nabla \rho \|_{L^2}^4 + \| \nabla \mathbf{u} \|_{L^2}^4 \right) \\
				&\le C(\varepsilon, T) \left( \| \sqrt{\rho}\dot{\mathbf{u}} \|_{L^2}^4 + \| \nabla \rho \|_{L^2}^2 + \| \nabla \mathbf{u} \|_{L^2}^2 \right),
			\end{aligned}
		\end{equation}
		and 
		\begin{equation}
			\begin{aligned}
				C(\varepsilon, T)\int_{\Omega} |\nabla \rho|^2 |\mathbf{u}|^2 |\nabla \mathbf{u}|^2 \,dx 
				&\le C(\varepsilon, T) \| \nabla \rho \|_{L^4}^2 \| \mathbf{u} \|_{L^4}^2 \| \nabla \mathbf{u} \|_{L^{\infty}}^2 \\
				&\le C(\varepsilon, T) \| \nabla \mathbf{u} \|_{W^{1,4}}^2 \\
				&\le C(\varepsilon, T) \left( \| \sqrt{\rho}\dot{\mathbf{u}} \|_{L^2}^{\frac{4-N}{2}} \| \nabla \dot{\mathbf{u}} \|_{L^2}^{\frac{N}{2}} + \| \nabla \rho \|_{L^4}^2 + \| \nabla \mathbf{u} \|_{L^2}^2 + 1 \right) \\
				&\le \varepsilon \| \nabla \dot{\mathbf{u}} \|_{L^2}^2 + C(\varepsilon, T) \left( \| \sqrt{\rho}\dot{\mathbf{u}} \|_{L^2}^2 + \| \nabla \rho \|_{L^4}^2 + \| \nabla \mathbf{u} \|_{L^2}^2 + 1 \right).
			\end{aligned}
		\end{equation}
		Therefore, $N_2$ can be estimated by 
		\begin{equation}\label{66.14}
			\begin{aligned}
				N_2 &\le -2\int_{\Omega} \rho |\nabla \dot{\mathbf{u}}|^2 \,dx + 3\varepsilon \| \nabla \dot{\mathbf{u}} \|_{L^2}^2 + C(\varepsilon, T) \| \sqrt{\rho}\dot{\mathbf{u}} \|_{L^2}^4 \\
				&\quad + C(\varepsilon, T)\left( \| \sqrt{\rho}\dot{\mathbf{u}} \|_{L^2}^2 + \| \nabla \rho \|_{L^4}^2 + \| \nabla \mathbf{u} \|_{L^2}^2 + 1 \right).
			\end{aligned}
		\end{equation}
		Combining \eqref{66.68}, \eqref{66.69} and \eqref{66.14}, we arrive at 
		\begin{equation}
			\begin{aligned}
				\left( \frac{1}{2}\int_{\Omega} \rho |\dot{\mathbf{u}}|^2 \,dx \right)_t &+ 2C(T)^{-1} \int_{\Omega} |\nabla \dot{\mathbf{u}}|^2 \,dx - 4\varepsilon \| \nabla \dot{\mathbf{u}} \|_{L^2}^2 \\
				&\le \left( \frac{1}{2}\int_{\Omega} \rho |\dot{\mathbf{u}}|^2 \,dx \right)_t + 2\int_{\Omega} \rho |\nabla \dot{\mathbf{u}}|^2 \,dx - 4\varepsilon \| \nabla \dot{\mathbf{u}} \|_{L^2}^2 \\
				&\le C(\varepsilon, T) \left( \| \sqrt{\rho}\dot{\mathbf{u}} \|_{L^2}^2 + \| \nabla \rho \|_{L^4}^2 + \| \nabla \mathbf{u} \|_{L^2}^2 + 1 \right) \\
				&\quad + C(\varepsilon, T) \| \sqrt{\rho}\dot{\mathbf{u}} \|_{L^2}^4.
			\end{aligned}
		\end{equation}
		Let $\varepsilon =\frac{C\left( T \right) ^{-1}}{4}$, then using Gronwall's inequality with \eqref{666.2} and \eqref{66.2}, we arrive at \eqref{666.7}.
	\end{proof}
	\begin{proposition}
		Assume that \eqref{66.1} holds. Then there exists a constant $C(T) > 0$ such that
		\begin{equation}\label{666.118}
			\sup_{0 \le t \le T} (\|\rho\|_{H^2}^2+\|s\|_{H^2}^2) + \int_0^T \|\mathbf{u}\|_{H^3}^2 \, dt \le C(T).
		\end{equation}
	\end{proposition}
	\begin{proof}
		Applying $\partial_i\partial_j$ to equation \eqref{1.1}$_1$, we obtain
		$$\begin{aligned}
			\partial_t(\partial_{ij}\rho) + \partial_{ij}\mathbf{u}^k\partial_k\rho &+ \partial_i\mathbf{u}^k\partial_{jk}\rho + \partial_j\mathbf{u}^k\partial_{ik}\rho + \mathbf{u}^k\partial_{ijk}\rho + \partial_{ij}\rho\partial_k\mathbf{u}^k 
			\\
			&+ \partial_i\rho\partial_{jk}\mathbf{u}^k + \partial_j\rho\partial_{ik}\mathbf{u}^k + \rho\partial_{ijk}\mathbf{u}^k = 0.
		\end{aligned}$$
		Multiplying by $\partial_{ij}\rho$, summing over $i, j$, and integrating over $\Omega$, we utilize integration by parts to obtain
		\begin{equation}\label{66.18}
			\begin{aligned}
				\frac{d}{dt}\int_{\Omega}|\nabla ^2\rho |^2\,dx &\le C\left( \int_{\Omega}|\nabla \rho ||\nabla ^2\mathbf{u}||\nabla ^2\rho |\,dx+\int_{\Omega}|\nabla \mathbf{u}||\nabla ^2\rho |^2\,dx+\int_{\Omega}\rho|\nabla ^3\mathbf{u}||\nabla ^2\rho |\,dx \right)\\
				&\le C(T)(1+\|\nabla \mathbf{u}\|_{L^{\infty}})\int_{\Omega}|\nabla ^2\rho |^2\,dx\\
				&\quad +C(T)\int_{\Omega}|\nabla \rho |^2|\nabla ^2\mathbf{u}|^2\,dx+C(T)\int_{\Omega}|\nabla ^3\mathbf{u}|^2\,dx\\
				&\le C(T)(1+\|\nabla \mathbf{u}\|_{L^{\infty}})\|\nabla ^2\rho \|_{L^2}^{2}\\
				&\quad +C(T)\|\nabla \rho \|_{L^4}^{2}\|\nabla ^2\mathbf{u}\|_{L^4}^{2}+C(T)\|\nabla ^3\mathbf{u}\|_{L^2}^{2}\\
				&\le C(T)(1+\|\nabla \mathbf{u}\|_{L^{\infty}})\|\nabla ^2\rho \|_{L^2}^{2}+C(T)+C(T)\|\mathbf{u}\|_{H^3}^{2},
			\end{aligned}
		\end{equation}
		Since standard $H^1$ estimates for elliptic systems \eqref{66.55} yields
		\begin{equation}\label{66.19}
			\begin{aligned}
				\|\mathbf{u}\|_{H^3}^{2} &\le C\|\rho ^{-1}(\rho \dot{\mathbf{u}}+\nabla P-2\nabla \rho \cdot \nabla \mathbf{u})\|_{H^1}^{2}\\
				&\le C(T)\left( \|\sqrt{\rho}\dot{\mathbf{u}}\|_{L^2}+\|\nabla \rho \|_{L^2}+1+\|\nabla \mathbf{u}\|_{L^2} \right)\\
				&\quad +C(T)\left( \|\nabla \dot{\mathbf{u}}\|_{L^2}^{2}+\|\nabla \rho \|_{L^4}^{4}+\|\nabla \rho \|_{L^4}^{2}\|\nabla s\|_{L^4}^{2}+\|\nabla ^2\rho \|_{L^2}^{2} \right)\\
				&\quad +C(T)\left( \|\nabla s\|_{L^4}^{4}+\|\nabla ^2s\|_{L^2}^{2}+\| |\nabla \rho |^2|\nabla \mathbf{u}| \|_{L^2}^{2} \right)\\
				&\quad +C(T)\left( \| |\nabla ^2\rho ||\nabla \mathbf{u}| \|_{L^2}^{2}+\| |\nabla \rho ||\nabla ^2\mathbf{u}| \|_{L^2}^{2} \right)\\
				&\le C(T)\left( 1+\|\nabla \dot{\mathbf{u}}\|_{L^2}^{2}+\|\nabla ^2\rho \|_{L^2}^{2}+\|\nabla ^2s\|_{L^2}^{2} \right)\\
				&\quad +C(T)\|\nabla ^2\rho \|_{L^2}^{2}\|\nabla \mathbf{u}\|_{W^{1,4}}^{2}+C(M)\|\nabla ^2\mathbf{u}\|_{L^4}^{2}\\
				&\le C(T)\left( 1+\|\nabla \dot{\mathbf{u}}\|_{L^2}^{2} \right) \left( 1+\|\nabla ^2\rho \|_{L^2}^{2} \right) +C(T)\|\nabla ^2s\|_{L^2}^{2}\\
				&\quad +C(T)\|\nabla ^2\mathbf{u}\|_{L^2}^{\frac{4-N}{2}}\|\nabla ^2\mathbf{u}\|_{H^1}^{\frac{N}{2}}\\
				&\le \frac{1}{2}\|\mathbf{u}\|_{H^3}^{2}+C(T)\left( 1+\|\nabla \dot{\mathbf{u}}\|_{L^2}^{2} \right) \left( 1+\|\nabla ^2\rho \|_{L^2}^{2} \right) +C(T)\|\nabla ^2s\|_{L^2}^{2}.
			\end{aligned}
		\end{equation}
		Applying the operator $\partial_{ij}$ to the equation $s_t + \mathbf{u} \cdot \nabla s = 0$, multiplying by $\partial_{ij}s$, summing over $i, j$, and integrating over $\Omega$, we obtain
		\begin{equation}\label{66.20}
			\begin{aligned}
				\frac{1}{2}\frac{d}{dt}\int_{\Omega}|\nabla ^2s|^2\,dx &=-\int_{\Omega}(\partial _{ij}s)\left( \partial _{ij}\mathbf{u}^k\partial _ks+\partial _i\mathbf{u}^k\partial _{jk}s+\partial _j\mathbf{u}^k\partial _{ik}s \right) \,dx\\
				&\quad -\int_{\Omega}(\partial _{ij}s)\mathbf{u}^k\partial _k(\partial _{ij}s)\,dx\\
				&=-\int_{\Omega}\partial _{ij}s\left( \partial _{ij}\mathbf{u}\cdot \nabla s+2\nabla \mathbf{u}\cdot \nabla ^2s \right) \,dx+\frac{1}{2}\int_{\Omega}(\mathrm{div}\mathbf{u})|\nabla ^2s|^2\,dx\\
				&\le C\|\nabla \mathbf{u}\|_{L^{\infty}}\|\nabla ^2s\|_{L^2}^{2}+C\|\nabla s\|_{L^{\infty}}\|\nabla ^2\mathbf{u}\|_{L^2}\|\nabla ^2s\|_{L^2}\\
				&\le C\left( \|\nabla \mathbf{u}\|_{L^{\infty}}+1 \right) \|\nabla ^2s\|_{L^2}^{2}+C\|\nabla ^2\mathbf{u}\|_{L^2}^{2}.
			\end{aligned}
		\end{equation}
		It follows from \eqref{666.11} that 
		\begin{equation}\label{66.22}
			\|\nabla \mathbf{u}\|_{L^{\infty}} \le C\|\nabla \mathbf{u}\|_{W^{1,4}} \le C(T)\left( \|\nabla \dot{\mathbf{u}}\|_{L^2}^{2}+1 \right).
		\end{equation}
		Combining \eqref{66.18}-\eqref{66.22}, we arrive at 
		\begin{equation}
			\begin{aligned}
				\frac{d}{dt}\int_{\Omega}\left( |\nabla ^2s|^2+|\nabla ^2\rho |^2 \right)\,dx
				&\le C(T) \left( \|\nabla \dot{\mathbf{u}}\|_{L^2}^{2}+1 \right) \left( \|\nabla ^2s\|_{L^2}^{2}+\|\nabla ^2\rho \|_{L^2}^{2} \right) \\
				&\quad + C(T) + \|\nabla \dot{\mathbf{u}}\|_{L^2}^{2}.
			\end{aligned} 
		\end{equation}
		Using Gronwall’s inequality, \eqref{666.7} and \eqref{66.19}, we get \eqref{666.118}.
	\end{proof}
	\begin{proposition}
		Assume that \eqref{66.1} holds. Then, there exists a constant $C(T)$ such that
		\begin{equation} \label{666.16}
			\sup_{0 \le t \le T} \left( \|\nabla \dot{\mathbf{u}}\|_{L^2}^2 + \|\nabla \mathbf{u}_t\|_{L^2}^2 + \|\mathbf{u}\|_{H^3}^2 \right) + \int_0^T \left( \|\mathbf{u}_{tt}\|_{L^2}^2 + \|\mathbf{u}_t\|_{H^2}^2 \right) dt \le C(T).
		\end{equation}
	\end{proposition}
	\begin{proof}
		Differentiating \eqref{66.55} with respect to $t$ gives
		\begin{equation} \label{66.25}
			-\Delta u_t=-\left[ \frac{1}{\rho}(\rho \dot{u}+\nabla P-2\nabla \rho \cdot \nabla u) \right] _t.
		\end{equation}
		Taking the $L^2$ inner product with $u_{tt}$ and integrating by parts, we derive
		\begin{equation} \label{66.26}
			\begin{aligned}
				\frac{1}{2}&\frac{d}{dt}\int_{\Omega}|\nabla \mathbf{u}_t|^2 \,dx + \int_{\Omega}|\mathbf{u}_{tt}|^2\,dx\\
				&\le C(T)\int_{\Omega}|\mathbf{u}_{tt}|\Big(|\mathbf{u}_t||\nabla \mathbf{u}|+|\mathbf{u}||\nabla \mathbf{u}_t|+|\rho _t||\nabla \rho |+|\nabla \rho _t|\\
				&\quad +|s_t||\nabla \rho|+|\rho _t||\nabla s| +|\nabla s_t|+|s_t||\nabla s|\\
				&\quad +|\rho _t||\nabla \rho ||\nabla \mathbf{u}|+|\nabla \rho _t||\nabla \mathbf{u}|+|\nabla \rho ||\nabla \mathbf{u}_t|\Big)\,dx\\
				&\le \frac{1}{2}\int_{\Omega}|\mathbf{u}_{tt}|^2\,dx + C(T)\int_{\Omega}\Big(|\mathbf{u}_t|^2|\nabla \mathbf{u}|^2+|\mathbf{u}|^2|\nabla \mathbf{u}_t|^2+|\rho _t|^2|\nabla \rho |^2\\
				&\quad +|\nabla \rho _t|^2+|s_t|^2|\nabla \rho|^2 +|\rho _t|^2|\nabla s|^2+|\nabla s_t|^2+|s_t|^2|\nabla s|^2\\
				&\quad +|\rho _t|^2|\nabla \rho |^2|\nabla \mathbf{u}|^2+|\nabla \rho _t|^2|\nabla \mathbf{u}|^2+|\nabla \rho |^2|\nabla \mathbf{u}_t|^2\Big)\,dx\\
				&\le \frac{1}{2}\int_{\Omega}|\mathbf{u}_{tt}|^2\,dx + C(T)\Big(\|\mathbf{u}_t\|_{L^6}^{2}\|\nabla \mathbf{u}\|_{L^6}^{2}+\|\mathbf{u}\|_{L^{\infty}}^{2}\|\nabla \mathbf{u}_t\|_{L^2}^{2}\\
				&\quad +\|\rho _t\|_{L^6}^{2}\|\nabla \rho \|_{L^6}^{2} +\|\nabla \rho _t\|_{L^2}^{2}+\|s_t\|_{L^6}^{2}\|\nabla \rho \|_{L^6}^{2}\\
				&\quad +\|\rho _t\|_{L^6}^{2}\|\nabla s\|_{L^6}^{2}+\|\nabla s_t\|_{L^2}^{2}+\|s_t\|_{L^6}^{2}\|\nabla s\|_{L^6}^{2}\\
				&\quad +\|\rho _t\|_{L^6}^{2}\|\nabla \rho \|_{L^6}^{2}\|\nabla \mathbf{u}\|_{L^6}^{2}+\|\nabla \rho _t\|_{L^2}^{2}\|\nabla \mathbf{u}\|_{L^{\infty}}^{2}+\|\nabla \rho \|_{L^6}^{2}\|\nabla \mathbf{u}_t\|_{L^3}^{2}\Big).
			\end{aligned}
		\end{equation}
		We note that 
		\begin{equation}\label{66.27}
			\begin{aligned}
				\|\nabla \mathbf{u}_t\|_{L^2}^{2} &\le C\|\nabla \dot{\mathbf{u}}\|_{L^2}^{2}+C\|\nabla (\mathbf{u}\cdot \nabla \mathbf{u})\|_{L^2}^{2}\le C\|\nabla \dot{\mathbf{u}}\|_{L^2}^{2}+C(T),\\
				\|\nabla \dot{\mathbf{u}}\|_{L^2}^{2} &\le C\|\nabla \mathbf{u}_t\|_{L^2}^{2}+C\|\nabla (\mathbf{u}\cdot \nabla \mathbf{u})\|_{L^2}^{2}\le C\|\nabla \mathbf{u}_t\|_{L^2}^{2}+C(T),
			\end{aligned}
		\end{equation}
		and the $L^2$-estimates of time derivatives for the elliptic system \eqref{66.55} 
		\begin{equation}\label{66.28}
			\begin{aligned}
				\|\mathbf{u}_t\|_{H^2}^{2} &\le C\left\| \left[ \rho ^{-1}(\rho \dot{\mathbf{u}}+\nabla P-2\nabla \rho \cdot \nabla \mathbf{u}) \right] _t \right\| _{L^2}^{2}\\
				&\le C(T)\Big(\|\mathbf{u}_{tt}\|_{L^2}^{2}+\| |\mathbf{u}_t||\nabla \mathbf{u}|\|_{L^2}^{2}+\| |\mathbf{u}||\nabla \mathbf{u}_t|\|_{L^2}^{2}+\| \rho _t|\nabla \rho |\|_{L^2}^{2}\\
				&\quad +\|\nabla \rho _t\|_{L^2}^{2}+\| |s_t||\nabla \rho |\|_{L^2}^{2}+\|\nabla s_t\|_{L^2}^{2}+\| |\nabla s||s_t| \|_{L^2}^{2}+\| |\nabla s||\rho _t| \|_{L^2}^{2}\\
				&\quad +\|\rho _t|\nabla \rho ||\nabla \mathbf{u}|\|_{L^2}^{2}+\| |\nabla \rho _t||\nabla \mathbf{u}|\|_{L^2}^{2}+\| |\nabla \rho ||\nabla \mathbf{u}_t|\|_{L^2}^{2}\Big)\\
				&\le C(T)\Big(\|\mathbf{u}_{tt}\|_{L^2}^{2}+\|\mathbf{u}_t\|_{L^6}^{2}\|\nabla \mathbf{u}\|_{L^6}^{2}+\|\mathbf{u}\|_{L^{\infty}}^{2}\|\nabla \mathbf{u}_t\|_{L^2}^{2}\\
				&\quad +\|\rho _t\|_{L^6}^{2}\|\nabla \rho \|_{L^6}^{2}+\|\nabla \rho _t\|_{L^2}^{2}+\|s_t\|_{L^6}^{2}\|\nabla \rho \|_{L^6}^{2}+\|\nabla s_t\|_{L^2}^{2}\\
				&\quad +\|\nabla s\|_{L^6}^{2}\|s_t\|_{L^6}^{2}+\|\nabla s\|_{L^6}^{2}\|\rho _t\|_{L^6}^{2}+\|\rho _t\|_{L^6}^{2}\|\nabla \rho \|_{L^6}^{2}\|\nabla \mathbf{u}\|_{L^6}^{2}\\
				&\quad +\|\nabla \rho _t\|_{L^2}^{2}\|\nabla \mathbf{u}\|_{L^{\infty}}^{2}+\|\nabla \rho \|_{L^6}^{2}\|\nabla \mathbf{u}_t\|_{L^3}^{2}\Big)\\
				&\le C(T)\left( 1+\|\nabla \dot{\mathbf{u}}\|_{L^2}^{2}+\|\mathbf{u}_{tt}\|_{L^2}^{2}+\|\nabla \mathbf{u}_t\|_{L^3}^{2} \right)\\
				&\le C(T)\left( 1+\|\nabla \dot{\mathbf{u}}\|_{L^2}^{2}+\|\mathbf{u}_{tt}\|_{L^2}^{2} \right) +C(T)\|\nabla \mathbf{u}_t\|_{L^2}^{\frac{6-N}{3}}\|\nabla \mathbf{u}_t\|_{H^1}^{\frac{N}{3}}\\
				&\le \frac{1}{2}\|\nabla \mathbf{u}_t\|_{H^1}^{2}+C(T)\left( 1+\|\nabla \dot{\mathbf{u}}\|_{L^2}^{2}+\|\mathbf{u}_{tt}\|_{L^2}^{2} \right),
			\end{aligned}
		\end{equation}
		where we used the fact that 
		\begin{equation}\label{66.29}
			\| \rho_t \|_{L^6}^2+\|\nabla \rho _t\|_{L^2}^{2}+\| s_t \|_{L^6}^2+\|\nabla s_t\|_{L^2}^{2} \le C(T).
		\end{equation}
		Substituting \eqref{66.27} and \eqref{66.29} into \eqref{66.26} and using Young's inequality, we arrive at
		\begin{equation}
			\begin{aligned}
				\frac{d}{dt}\int_{\Omega}|\nabla \mathbf{u}_t|^2\,dx + \int_{\Omega}|\mathbf{u}_{tt}|^2\,dx
				&\le \frac{1}{2}\int_{\Omega}|\mathbf{u}_{tt}|^2\,dx + C(T)\Big(1+\|\nabla \mathbf{u}_t\|_{L^2}^{2} \\
				&\quad +\|\nabla \mathbf{u}_t\|_{L^2}^{\frac{6-N}{3}}\|\nabla \mathbf{u}_t\|_{H^1}^{\frac{N}{3}}\Big) \\
				&\le \frac{1}{2}\int_{\Omega}|\mathbf{u}_{tt}|^2\,dx + C(T)(1+\|\nabla \mathbf{u}_t\|_{L^2}^{2}).
			\end{aligned}
		\end{equation}
		Using Gronwall's inequality, \eqref{66.19}, \eqref{66.27} and \eqref{66.28}, we obtain \eqref{666.16}.
	\end{proof}
	\begin{proposition}
		Assume that \eqref{66.1} holds. Then, there exists a constant $C(T)$ such that
		\begin{equation}\label{666.31}
			\begin{aligned}
				\sup_{0\le t\le T}(\|\rho\|_{H^3}^2+\|s\|_{H^3}^2+\|\rho_t\|_{H^2}^2 + \|\rho_{tt}\|_{L^2}^2+\|s_t\|_{H^2}^2 + \|s_{tt}\|_{L^2}^2) 
				\\
				+ \int_{0}^{T}\|\mathbf{u}\|_{H^4}^2 +\|\rho_{tt} \|_{H^1}^2+\|s_{tt} \|_{H^1}^2\,dt \le C(T).
			\end{aligned}
		\end{equation}
	\end{proposition}
	\begin{proof}
		Applying the operator $\partial_i\partial_j\partial_k$ to equation \eqref{1.1}$_1$, multiplying the resulting equation by $\partial_{ijk}\rho$, summing $i,j,k$ and integrating by parts, we derive 
		\begin{equation}\label{66.32}
			\begin{aligned}
				\frac{1}{2}\frac{d}{dt}\int_{\Omega}|\nabla^3\rho|^2\,dx
				&\le C\int_{\Omega}|\nabla^3\rho|\Big(\rho|\nabla^4\mathbf{u}| + |\nabla\rho||\nabla^3\mathbf{u}| + |\nabla^2\rho||\nabla^2\mathbf{u}| + |\nabla^3\rho||\nabla\mathbf{u}|\Big)\,dx \\
				&\le C(T)\Big(\|\nabla^3\rho\|_{L^2}^2 + \|\nabla^4\mathbf{u}\|_{L^2}^2 + \|\nabla\rho\|_{L^\infty}^2\|\nabla^3\mathbf{u}\|_{L^2}^2 \\
				&\quad + \|\nabla^2\rho\|_{L^6}^2\|\nabla^2\mathbf{u}\|_{L^6}^2 + \|\nabla^3\rho\|_{L^2}^2\|\nabla\mathbf{u}\|_{L^\infty}^2\Big) \\
				&\le C(T)(1+\|\nabla^3\rho\|_{L^2}^2+\|\nabla^4\mathbf{u}\|_{L^2}^2).
			\end{aligned}
		\end{equation}
		Using the $H^2$ estimates for elliptic system \eqref{66.55}, we obtain 
		\begin{equation}\label{66.33}
			\begin{aligned}
				\|\mathbf{u}\|_{H^4}^{2} &\le C\left\| \nabla^2\left( \rho^{-1}(\rho \mathbf{u}_t+\rho \mathbf{u}\cdot \nabla \mathbf{u}+\nabla P-2\nabla \rho \cdot \nabla \mathbf{u}) \right) \right\|_{L^2}^{2} + C(T)\\
				&\le C(T) + C(T)\Big( \|\nabla^2 \mathbf{u}_t\|_{L^2}^{2} + \| |\nabla^2 \mathbf{u}||\nabla \mathbf{u}| \|_{L^2}^{2} + \| |\mathbf{u}||\nabla^3 \mathbf{u}| \|_{L^2}^{2}\\
				&\quad + \|\nabla^3 \rho\|_{L^2}^{2} + \| |\nabla^2 \rho||\nabla \rho| \|_{L^2}^{2} + \| |\nabla^2 \rho||\nabla s| \|_{L^2}^{2} + \| |\nabla \rho|^3 \|_{L^2}^{2}\\
				&\quad + \| |\nabla \rho|^2|\nabla s| \|_{L^2}^{2} + \| |\nabla \rho||\nabla^2 s| \|_{L^2}^{2} + \| |\nabla \rho||\nabla s|^2 \|_{L^2}^{2}\\
				&\quad + \| |\nabla s||\nabla^2 s| \|_{L^2}^{2} + \|\nabla^3 s\|_{L^2}^{2} + \| |\nabla s|^3 \|_{L^2}^{2}\\
				&\quad + \| |\nabla \mathbf{u}||\nabla \rho|^3 \|_{L^2}^{2} + \| |\nabla \rho||\nabla^2 \rho||\nabla \mathbf{u}| \|_{L^2}^{2} + \| |\nabla^2 \mathbf{u}||\nabla \rho|^2 \|_{L^2}^{2}\\
				&\quad + \| |\nabla^2 \mathbf{u}||\nabla^2 \rho| \|_{L^2}^{2} + \| |\nabla^3 \mathbf{u}||\nabla \rho| \|_{L^2}^{2}+ \| |\nabla^3\rho| |\nabla u|\|_{L^2}^2 \Big)\\
				&\le C(T)\left( 1 + \|\nabla^3 \rho\|_{L^2}^{2} + \|\nabla^3 s\|_{L^2}^{2} + \|\mathbf{u}_t\|_{H^2}^{2} \right)\\
				&\le C(T)\left( 1 + \|\nabla^3 \rho\|_{L^2}^{2} + \|\nabla^3 s\|_{L^2}^{2} + \|\mathbf{u}_{tt}\|_{L^2}^{2} \right).
			\end{aligned} 
		\end{equation}
		Similarly, applying the operator $\partial_i\partial_j\partial_k$ to equation (1.1)$_3$, and multiplying by $\partial_{ijk}s$, we arrive at
		\begin{equation}\label{66.34}
			\begin{aligned}
				\frac{1}{2}\frac{d}{dt}\|\nabla^3 s\|_{L^2}^2 &= -\int_{\Omega} \partial_{ijk}(\mathbf{u} \cdot \nabla s) \partial_{ijk} s \,dx \\
				&= -\int_{\Omega} (\mathbf{u} \cdot \nabla \partial_{ijk} s) \partial_{ijk} s \,dx - \int_{\Omega} \Big( \partial_{ijk}(\mathbf{u} \cdot \nabla s) - \mathbf{u} \cdot \nabla \partial_{ijk} s \Big) \partial_{ijk} s \,dx \\
				&= \frac{1}{2}\int_{\Omega} (\operatorname{div} \mathbf{u}) |\nabla^3 s|^2 \,dx - \int_{\Omega} \Big( \partial_{ijk}(\mathbf{u} \cdot \nabla s) - \mathbf{u} \cdot \nabla \partial_{ijk} s \Big) \partial_{ijk} s \,dx \\
				&\le C \|\nabla \mathbf{u}\|_{L^\infty} \|\nabla^3 s\|_{L^2}^2 + C \int_{\Omega} \Big( |\nabla \mathbf{u}||\nabla^3 s| + |\nabla^2 \mathbf{u}||\nabla^2 s| + |\nabla^3 \mathbf{u}||\nabla s| \Big) |\nabla^3 s| \,dx \\
				&\le C(T) (\|\nabla \mathbf{u}\|_{L^\infty} +1)\|\nabla^3 s\|_{L^2}^2 + C(T)\|\nabla^3 \mathbf{u}\|_{L^2}^2.
			\end{aligned}
		\end{equation}
		Combining \eqref{66.32}-\eqref{66.34}, we get
		\begin{equation}
			\frac{d}{dt}\left( \| \nabla^3\rho \|_{L^2}^{2} + \| \nabla^3s \|_{L^2}^{2} \right) \le C(T) \left( 1 + \| \nabla^3\rho \|_{L^2}^{2} + \| \nabla^3s \|_{L^2}^{2} \right) + C(T) \| \mathbf{u}_{tt} \|_{L^2}^{2}.
		\end{equation}
		Using Gronwall's inequality and in view of the structure of the density and entropy equations, we finally arrive at \eqref{666.31}.
	\end{proof}
	\section{Proof of the Main Results}
	\subsection{Proof of the Theorem \ref{theorem 2.1} and \ref{theorem 2.2}}
	For $N=2,\gamma>1$ and $N=3, 1<\gamma<3$, based on lower-order estimates  
	\begin{equation}
		\begin{aligned}
			\mathop {\mathrm{sup}} \limits_{0\le t\le T}&\left( \| \mathbf{u}\| _{L^4} +\| \rho \| _{L^{\infty}}+\| \rho ^{-1}\| _{L^{\infty}}+\| \nabla \rho \| _{L^2}+\| \nabla \rho \| _{L^4} \right) 
			\\
			&+\int_0^T{\left( \| \nabla \rho \| _{L^2}^{2}+\| \nabla \rho \| _{L^4}^{4}+\|\nabla \mathbf{u}\| _{L^2}^{2} \right)}dt\le C(T),
		\end{aligned}
	\end{equation} 
	We derive the higher-order estimates
	\begin{equation}\label{666.22}
		\begin{aligned}
			\sup_{0\le t\le T}&\left( \|\rho\|_{H^3} + \|\rho_t\|_{H^2} + \|\mathbf{u}\|_{H^3} + \|\mathbf{u}_t\|_{H^1} + \|s\|_{H^3} + \|s_t\|_{H^2} \right) \\
			&+ \int_0^T \left( \|\mathbf{u}\|_{H^4}^{2} + \|\mathbf{u}_t\|_{H^2}^{2} + \|\mathbf{u}_{tt}\|_{L^2}^{2} \right) \,dt \le C(T).
		\end{aligned}
	\end{equation}
	Now that we have obtained the global a priori estimates, we proceed to prove that the solution can be extended to a global one.\\
	\textbf{1.Regularity} \par
	First, we verify the continuity of the solution, i.e.
	\begin{equation}\label{66.3}
		\rho ,s\in C\left( \left[ 0,T \right] ;H^3 \right) ,\quad u\in C\left( \left[ 0,T \right] ;H^3 \right) .
	\end{equation}
	The density regularity $\rho \in L^\infty(0, T; H^3)$ and $\rho_t \in L^\infty(0, T; H^2)$ ensure that $\rho \in C([0, T]; H^2)$. The strong continuity in $H^3$ is deduced via standard arguments involving weak and norm convergence. The regularity of $s$ is obtained similarly. For the velocity, the condition $\mathbf{u} \in L^2(0, T; H^4) \cap H^1(0, T; H^2)$ implies $\mathbf{u} \in C([0, T]; H^3)$ by standard interpolation theory.
	\\
	\textbf{2.Global Existence}\par
	To establish that $T^* = \infty$, we argue by contradiction. Suppose $T^* < \infty$. The estimates \eqref{66.2} remain valid as $t \to T^*$, continuity proposition \eqref{66.3} allowing us to define the limit $(\rho(T^*), s(T^*), \mathbf{u}(T^*)) \in H^3$. The density maintains a positive lower and upper bound. Moreover, based on the transport principle, we have $\frac{\partial_r s}{r^{N-1}\rho} \in L^\infty$. We may regard these limits as new initial data. Applying the local existence result (Lemma \ref{lemma2.1}) extends the solution to $[T^*, T^* + t_0]$ for some $t_0 > 0$. This contradicts the maximality of $T^*$. Thus, we complete the proof.
	\subsection{Proof of the Corollary \ref{cor1.1}}
	In fact, we have the following uniform bounds in this case
	$$
	\begin{aligned}
		\sup_{0\leq t<\infty}&\left(\|\rho\|_{L^\infty(\Omega)}+\|\nabla\rho^{\frac{1}{2}}\|_{L^2(\Omega)}+\|\nabla\rho^{\frac{1}{4}}\|_{L^{4}(\Omega)}\right)\\
		&+\int_0^\infty\|\nabla\rho^{\frac{\gamma}{2}}\|_{L^{2}(\Omega)}^2+\|\nabla\rho^{\frac{\gamma}{4}}\|_{L^{4}(\Omega)}^{4}+\|\rho^{\frac{1}{2}}\nabla \mathbf{u}\|_{L^2(\Omega)}^2\,dt\leq C,
	\end{aligned}
	$$ where $C$ is only dependent on initial data.
	These bounds ensure that the density converges uniformly to its equilibrium state, yielding a uniform positive lower bound for $\rho$ and the large-time behavior of $(\rho,\mathbf{u})$; see \cite{Huang-Meng-Zhang} for details.
	
	\section*{Acknowledgments}
	Xiangdi Huang is partially supported by Chinese Academy of Sciences Project for Young Scientists in Basic Research (Grant No. YSBR-031), National Natural Science Foundation of China (Grant Nos. 12494542, 11688101). The author is grateful to Dr. Weili Meng for providing helpful comments and suggestions that significantly improved this paper.
	
	\vspace{1cm}
	\noindent\textbf{Data availability statement.} Data sharing is not applicable to this article.
	
	\vspace{0.3cm}
	\noindent\textbf{Conflict of interest.} The authors declare that they have no conflict of interest.

	\bibliographystyle{siam}
	
\end{document}